\documentclass[11pt, leqno]{amsart}

\usepackage{geometry, amsmath, amssymb, amsthm, mathrsfs, setspace, comment, amscd, latexsym, mathtools, ascmac, url}
\usepackage[all]{xy}
\usepackage{hyperref}

\usepackage[parfill]{parskip}
\usepackage{parskip}

\usepackage[abs]{overpic} 

\usepackage{float}
\usepackage{tikz}
\usepackage{epstopdf}

\usepackage{subcaption}

\usepackage{color,graphicx}

\makeatletter
\@addtoreset{equation}{section}

\makeatother

\newtheorem{theorem}{Theorem}[section]
\newtheorem{lemma}[theorem]{Lemma}
\newtheorem{corollary}[theorem]{Corollary}
\newtheorem{proposition}[theorem]{Proposition}
\theoremstyle{definition}

\newtheorem{example}[theorem]{Example}
\newtheorem{question}[theorem]{Question}
\newtheorem{convention}[theorem]{Convention}

\newtheorem{remark}[theorem]{Remark}
\newtheorem{notation}[theorem]{Notation}

\newcommand{\del}{\partial}

\newcommand{\Z}{\mathbb{Z}}
\newcommand{\R}{\mathbb{R}}
\newcommand{\C}{\mathbb{C}}

\newcommand{\CP}{\mathbb{CP}}
\newcommand{\Int}{\textrm{Int}}
\newcommand{\M}{\mathcal{M}}

\newcommand{\PD}{\mathrm{PD}}

\newcommand{\p}{\partial}
\newcommand{\ow}{\omega}

\newcommand{\lda}{\lambda}

\renewcommand{\O}{\mathcal{O}}
\renewcommand{\P}{\mathbb{P}}

\definecolor{cadmiumgreen}{rgb}{0.0, 0.42, 0.24}
\definecolor{asparagus}{rgb}{0.53, 0.66, 0.42}
\definecolor{ao(english)}{rgb}{0.0, 0.5, 0.0}
\definecolor{amethyst}{rgb}{0.6, 0.4, 0.8}
\definecolor{blush}{rgb}{0.87, 0.36, 0.51}
\definecolor{brown(web)}{rgb}{0.65, 0.16, 0.16}
\definecolor{fandango}{rgb}{0.71, 0.2, 0.54}
\definecolor{fluorescentpink}{rgb}{1.0, 0.08, 0.58}
\definecolor{jazzberryjam}{rgb}{0.65, 0.04, 0.37}
\definecolor{lavenderpurple}{rgb}{0.59, 0.48, 0.71}
\definecolor{majorelleblue}{rgb}{0.38, 0.31, 0.86}
\definecolor{applegreen}{rgb}{0.55, 0.71, 0.0}
\definecolor{amber}{rgb}{1.0, 0.75, 0.0}
\definecolor{amber(sae/ece)}{rgb}{1.0, 0.49, 0.0}

\title{Rational ruled surfaces as symplectic hyperplane sections}

\author[Myeonggi Kwon]{Myeonggi Kwon}
\address{Department of Mathematics Education, Sunchon National University, Suncheon 57922, Republic of Korea}
\email{mkwon@scnu.ac.kr}

\author[Takahiro Oba]{Takahiro Oba}
\address{Department of Mathematics, Kyoto University, Kyoto 606-8502, Japan}
\curraddr{Department of Mathematics, Osaka University, Toyonaka, Osaka 560-0043, Japan}
\email{taka.oba@math.sci.osaka-u.ac.jp}

\date{\today}

\begin{document}

\maketitle

\begin{abstract}
We study embeddability of rational ruled surfaces as symplectic hyperplane sections into closed integral symplectic manifolds. From this we obtain results on Stein fillability of Boothby--Wang bundles over rational ruled surfaces. 

\end{abstract}


\section{Introduction}

The aim of this paper is to study the symplectic topology of 
certain symplectic hypersurfaces of closed 
integral symplectic manifolds, 
called symplectic hyperplane sections. 
By an integral symplectic manifold, we mean a symplectic manifold $(M, \omega)$ with 
$[\omega] \in H^{2}(M; \Z)$.
Motivated by the notion of ample divisors in complex geometry, we define 
a \textit{symplectic hyperplane section} of degree $k>0$ 
on an integral symplectic manifold $(M, \omega)$   
to be a symplectic submanifold $\Sigma \subset M$ of codimension $2$ 
such that 
the homology class $[\Sigma] \in H_{2}(M; \Z)$ is Poincar\'{e} dual to $k[\omega]$ and 
the complement $M \setminus \Sigma$ admits a Stein structure. 
A good source of symplectic hyperplane sections is Donaldson's construction \cite{Don}, which 
guarantees the existence of a symplectic hyperplane section of sufficiently large degree on a given closed integral symplectic manifold.

In complex geometry, where divisors have played an important role, rigidity aspects of complex manifolds are often captured 
by the existence of ample divisors. For example carrying an ample divisor forces the ambient space to be projective via Kodaira embedding.
The Lefschetz hyperplane theorem gives a strong restriction for a projective manifold 
to be an ample divisor. 
Sommese \cite{Somm} provided projective manifolds 
that cannot be embedded into any projective manifolds as ample divisors. 
See also \cite{Fuji} and \cite{Silva}.

Inspired by complex geometry it is interesting to study rigidity and flexibility aspects of symplectic manifolds in terms of the existence of symplectic hyperplane sections.
We can show that
every closed Riemann surface $(\Sigma, \ow)$ with symplectic volume 
$\int_{\Sigma} \omega \in \Z_{>0}$ can be embedded into a closed symplectic $4$-manifold as a symplectic hyperplane section. 
In dimension $4$, symplectic structures are more sensitive to embedability. 
The projective space $\CP^2$ with $\omega_{\textrm{FS}}$ is a symplectic hyperplane section on $(\CP^3, \omega_{\textrm{FS}})$ where $\omega_{\textrm{FS}}$ denotes the respective Fubini--Study form; 
The same space with multiple $k \omega_{\textrm{FS}}$ ($k \geq 2$) cannot be a symplectic hyperplane section on any integral symplectic $6$-manifold.
This is a consequence of Stein non-fillability of a contact structure on a certain $S^{1}$-bundle over $(\CP^{2}, k\omega_{\textrm{FS}})$ (see \cite{PP} for example).

In this paper we study the next simplest symplectic $4$-manifolds, namely \textit{rational ruled surfaces}, as symplectic hyperplane sections. There are two diffeomorphism types of them: 
the product bundle $S^{2} \times S^{2}$ and the non-trivial bundle $S^{2} \widetilde{\times} S^{2}$. 
We equip them with the symplectic forms $\omega_{a,b}$ and $\widetilde{\omega}_{a,b}$, respectively, for each $a,b \in \Z_{>0}$, characterized by the equations (\ref{eqn: omega}) and (\ref{eqn: omega_tilde}) in Sections \ref{sec: Symplectic structures on S2xS2} and \ref{sec: Symplectic structures on S2tilde{x}S2}, respectively.

\begin{theorem}\label{thm: embeddability}
Let $(S^{2} \times S^{2}, \omega_{a,b})$ and $(S^{2} \widetilde{\times} S^{2}, \widetilde{\omega}_{a,b} )$ 
be symplectic manifolds as above. 

\begin{enumerate}
\item\label{item: embedding} For any  $a \geq 1$ (resp. $a \geq 2$), 
the symplectic manifold $(S^{2} \times S^{2}, \omega_{a,1})$ (resp. $(S^{2} \widetilde{\times} S^{2}, \widetilde{\omega}_{a,1} )$) 
can be embedded into a closed integral symplectic $6$-manifold as a symplectic hyperplane section of degree $1$. 
 
\item\label{item: non-embedding} For any odd $a \geq 5$ (resp. $a \geq 7$), 
the symplectic manifold $(S^{2} \times S^{2}, \omega_{a,2})$ 
(resp. $(S^{2} \widetilde{\times} S^{2}, \widetilde{\omega}_{a,2} )$) 
cannot be embedded into any closed integral symplectic $6$-manifold as a symplectic hyperplane section 
of degree $1$. 
\end{enumerate}
\end{theorem}

We construct embeddings in Theorem \ref{thm: embeddability} from complex geometry. 
On the other hand the non-embeddability result comes from a holomorphic curve technique in symplectic geometry. 
We derive topological information on the complement of a symplectic hyperplane section 
from analyzing a moduli space of holomorphic spheres in the ambient closed symplectic manifold. 
This approach was studied by McDuff, Floer and Eliashberg \cite{McD} and 
has developed in several directions; see for example \cite{McD2}, \cite{Hind2}, \cite{OO}, \cite{Wen2}, \cite{BGZ} and the references therein. 
Our proof is mainly inspired by Hind \cite{Hind}. 

We would like to point out that B\u{a}descu \cite{Ba2, Ba} 
classified complex projective $3$-manifolds 
which contain rational ruled surfaces as ample divisors. 
In principle, his result gives the list of symplectic forms on rational ruled surfaces 
with which they cannot be symplectic hyperplane sections on any projective $3$-manifolds. 
Our non-embeddability result can be seen as an extension of this to symplectic hyperplane sections on symplectic $6$-manifolds. 

We apply Theorem \ref{thm: embeddability} to address the fillability problem of contact manifolds in light of \cite[Section 4.1.2]{BiranECM}.
Let $\Sigma$ be a symplectic hyperplane section on an integral symplectic manifold $(M,\omega)$.
Then the boundary of the complement $W$ of a tubular neighborhood of $\Sigma$ in $M$ carries a canonical principal $S^{1}$-bundle structure over $\Sigma$ with a contact structure. 
We call this bundle the \textit{Boothby--Wang bundle} over $\Sigma$ and its total space 
the \textit{Boothby--Wang manifold}
(see Section \ref{section: BW}).
Applying Theorem \ref{thm: embeddability} we prove 
the following fillability results. 

\begin{theorem}\label{thm: fillability}
Let $(P_{a,b}, \xi_{a,b})$ (resp. $(\widetilde{P}_{a,b}, \widetilde{\xi}_{a,b})$) 
be the Boothby--Wang manifold over $(S^{2} \times S^{2}, \omega_{a,b})$ 
(resp. $(S^{2} \widetilde{\times} S^{2}, \widetilde{\omega}_{a,b} )$ ). 

\begin{enumerate}
\item\label{item: critically fillable} The contact structure $\xi_{1,1}$ is critically Stein fillable. 

\item\label{item: subcritically fillable} The contact structures $\xi_{a,1}$ and $\widetilde{\xi}_{a,1}$ for $a \geq 2$ 
are subcritically Stein fillable.

\item\label{item: non-fillable} The contact structure $\xi_{a,2}$ for any odd $a \geq 5$ (resp. $\widetilde{\xi}_{a,2}$ for any odd $a \geq 7$) is not Stein fillable. 

\end{enumerate}
\end{theorem}

In the theorem, a Stein fillable contact structure $\xi$ on $P$ is said to be \textit{subcritically Stein fillable} if 
$(P, \xi)$ admits a subcritical Stein filling; 
otherwise, we call it \textit{critically Stein fillable}. See  Remarks \ref{rem: notsubcriticalfillable} and \ref{rem: almostWeinconstrs} for further statements on fillability.

It is known that the contact structures $\xi_{a,b}$ and $\xi_{a',b'}$ (resp. $\widetilde{\xi}_{a,b}$ and $\widetilde{\xi}_{a',b'}$) are equivalent as almost contact structures if and only if $a-b=a'-b'$ (resp. $2a-3b=2a'-3b'$) (see Corollary \ref{cor: equivalencefoalmostcontactstructures}). 
Lerman \cite[Question 1]{Ler} asked whether 
they are actually contactomorphic or not when their almost contact structures are equivalent. 
Theorem \ref{thm: fillability} gives a partial negative answer to his question from the fillability viewpoint. 
We would like to refer to the result in Boyer--Pati \cite{BP} which proves using contact homology that if $(a,b) \neq (a',b')$, they are not contactomorphic.

It is worth emphasizing that our non-fillability result is symplectic in nature. It is known that there are Stein non-fillable contact structures on $5$-manifolds (see \cite{BCS}, \cite{EKP} and \cite{PP} for example): In fact, these known $5$-manifolds cannot admit Stein fillable contact structures 
for topological reasons.
In contrast, our $5$-manifolds $P_{a,b}$ and $\widetilde{P}_{a,b}$, which are diffeomorphic to 
the trivial $S^{3}$-bundle 
$S^{2} \times S^{3}$ and the non-trivial one $S^{2} \widetilde{\times} S^{3}$ over $S^{2}$  respectively (see Proposition \ref{prop: BW}), carry Stein fillable contact structures. 
Moreover, each Stein non-fillable contact structure in the theorem is equivalent to a subcritically Stein fillable contact structure as almost contact structures.

This paper is organized as follows: 
Section \ref{section: embeddability} deals with the embedding problem of symplectic rational ruled surfaces. 
We first describe rational ruled surfaces as complex surfaces and construct symplectic forms on them in Section \ref{section: rational ruled surfaces}. 
After this, we give proofs of Theorem \ref{thm: embeddability} (\ref{item: embedding}) and (\ref{item: non-embedding}) in 
Section \ref{section: embedding} and \ref{section: non-embeddability}, respectively. 
Applying Theorem \ref{thm: embeddability}, we discuss the fillability problem in Section \ref{section: fillability}. 
Section \ref{section: BW} is devoted to reviewing the topology of Boothby--Wang bundles over rational ruled surfaces. 
A proof of Theorem \ref{thm: fillability} is provided in Sections \ref{section: fillable} and \ref{section: non-fillable}.
Finally, we conclude this paper by answering Lerman's question in Section \ref{section: almost contact}.

\begin{convention}\label{convention}
Since we are primarily interested in symplectic hyperplane sections of degree $1$, symplectic hyperplane sections in this paper are assumed to be of degree $1$ unless otherwise noted.
\end{convention}

\begin{remark}
Although the term \textit{symplectic hyperplane section} is used in some papers (see for example \cite{BiKh13} and \cite{DL}), 
definitions may slightly vary from paper to paper: 
Biran and Khanevsky \cite{BiKh13} impose nothing on its complement; 
Diogo and Lisi \cite{DL} do not require the cohomology class $[\omega]$ to have an integral lift. 
In this comparison, our definition is stronger than theirs. 
There is also a related notion called a \textit{smoothly polarized K\"ahler manifold},  which is a tuple of a K\"ahler manifold $(M, \omega, J)$ with integral $\omega$ and a complex hypersurface $\Sigma$ with $\mathrm{PD}[\Sigma]=k[\omega]$ for some positive integer $k$; see \cite{Bi01}.
\end{remark}

\subsection*{Acknowledgements} The authors would like to thank Otto van Koert and Zhengyi Zhou for helpful comments, and Akihiro Kanemitsu and Ryo Yamagishi for enlightening conversations on algebraic geometry. The first author was supported by the SFB/TRR 191 \emph{Symplectic Structures in Geometry, Algebra and Dynamics} funded by the DFG and by the National Research Foundation of Korea(NRF) grant funded by the Korea government(MSIT) (No. NRF-2021R1F1A1060118). The second author was supported by Japan Society for the Promotion of Science KAKENHI Grant Number 18J01373.

\section{Embedding problem of rational ruled surfaces}\label{section: embeddability}

\subsection{Rational ruled surfaces}\label{section: rational ruled surfaces}
We consider smooth $4$-manifolds which fiber over the $2$-sphere $S^{2}$ with fibers 
diffeomorphic to $S^{2}$. 
Motivated by complex geometry, we call such $4$-manifolds \textit{rational ruled surfaces}. 
It is known that there are only two $S^2$-bundle over $S^2$ (e.g. see \cite[Lemma 6.2.3]{McS}). In this subsection, we describe symplectic structures on them parametrized by pairs of two integers.

\subsubsection{Symplectic structures on $S^2 \times S^2$}\label{sec: Symplectic structures on S2xS2}

Let $\mathrm{pr}_i \colon S^2 \times S^2 \rightarrow S^2$ denote the $i$-th projection and $\omega_0$ the standard area form on $S^2$ normalized so that $\int_{S^2} \omega_0=1$. 
For positive integers $a$ and $b$, we define a symplectic form $\omega_{a,b}$ by 
\begin{equation}\label{eqn: omega}
	\omega_{a,b}=a\mathrm{pr}_{1}^{*}\omega_0+b\mathrm{pr}_{2}^{*}\omega_0. 
\end{equation}
This form is K\"ahler with respect to the standard product complex structure on $S^2 \times S^2$.

\subsubsection{Symplectic structures on $S^2 \widetilde{\times} S^2$}\label{sec: Symplectic structures on S2tilde{x}S2}

Next we describe symplectic structures on the non-trivial bundle $S^2 \widetilde{\times} S^2$. 
For later discussions, these symplectic forms need to be K\"ahler. 
Hence, identifying $S^2 \widetilde{\times} S^2$ with the complex manifold $\CP^2 \# \overline{\CP}^2$, we will present K\"ahler forms on $\CP^2 \# \overline{\CP}^2$ rather than $S^2 \widetilde{\times} S^2$. 

An $S^2$-bundle structure on $\CP^2 \# \overline{\CP}^2$ can be seen as follows (see also \cite[Example 7.1.5]{McS}): 
Let $f_0, f_1 \in \C[z_0, z_1, z_2]$ be generic homogeneous polynomials of degree $1$. 
These polynomials yield a family of hyperplanes 
$
	\{ [z] \in \CP^2 \mid \mu f_0(z)+\lambda f_1(z)=0 \}
$ 
parametrized by $[\lambda : \mu] \in \CP^1$, which intersect at one point, $x_0$. 
Thus, we obtain the map $p \colon \CP^{2} \setminus \{ x_0 \} \rightarrow \CP^1$ defined by 
$$
	p(z)=[f_0(z): f_1(z)]. 
$$
Now blowing up the point $x_0$ of $\CP^2$, we find that the map $p$ turns to a $\CP^1$-bundle over $\CP^1$, denoted by $\tilde{p} \colon \CP^2 \# \overline{\CP}^2 \rightarrow \CP^1$.
Write $\pi \colon \CP^2 \# \overline{\CP}^2 \rightarrow \CP^2$ for the blowing-up map: 
$$
  \xymatrix{
    & \CP^2 \# \overline{\CP}^2 \ar[ld]_{\tilde{p}} \ar[rd]^{\pi} & \\
    \CP^1 & & \CP^2.
  }
$$
Note that the maps $\pi$ and $\tilde{p}$ are $(J, J_{\textrm{st}})$-holomorphic, where  $J$ and $J_{\textrm{st}}$ denote the standard complex structures on $\CP^{2} \# \overline{\CP}^2$ and $\CP^{n}$ for $n=1,2$. 

For integers $a, b>0$ with $a>b$, we define the $2$-form $\widetilde{\omega}_{a,b}$ on $\CP^2 \# \overline{\CP}^2  (\approx S^2 \widetilde{\times} S^2)$ by 
\begin{equation}\label{eqn: omega_tilde}
	\widetilde{\omega}_{a,b}=(a-b)\tilde{p}^{*}\omega_0+a\pi^{*}\omega_{\textrm{FS}}.
\end{equation}
Here $\omega_{\textrm{FS}}$ denotes the Fubini--Study form on $\CP^2$ normalized so that $\int_{\CP^2}\omega_{\textrm{FS}}=1$.

\begin{lemma}\label{lemma: Kahler}
The $2$-form $\widetilde{\omega}_{a,b}$ is a K\"ahler form on $(\CP^2 \# \overline{\CP}^2, J)$.
\end{lemma}

\begin{proof}
Since $\omega_{\textrm{FS}}$ and $\omega_0$ are closed, so is $\widetilde{\omega}_{a,b}$. 
To see the non-degeneracy of $\widetilde{\omega}_{a,b}$, fix a point $x \in \CP^2 \# \overline{\CP}^2$ and take a non-zero tangent vector $v \in T_{x}(\CP^2 \# \overline{\CP}^2)$. 
The condition that $v \neq 0$ implies that $\pi_{*}v \neq 0$ or $\tilde{p}_{*} v \neq 0$. 
Hence, 
\begin{eqnarray*}
	\widetilde{\omega}_{a,b}(v, Jv) & = & (a-b)\tilde{p}^{*}\omega_0(v,Jv)+a\pi^{*}\omega_{\textrm{FS}}(v,Jv) \\
	& = & (a-b)\omega_0(\tilde{p}_{*}v, \tilde{p}_{*}(Jv))+a\omega_{\textrm{FS}}(\pi_{*}v,\pi_{*}(Jv)) \\
	& = & (a-b)\omega_0(\tilde{p}_{*}v, J_{\textrm{st}}(\tilde{p}_{*}v))+a\omega_{\textrm{FS}}(\pi_{*}v,J_{\textrm{st}}(\pi_{*}v)) \\
	& > & 0,
\end{eqnarray*}
where the third equality follows from the $(J,J_{\textrm{st}})$-holomorphicity of $\pi$ and $\tilde{p}$ and the last inequality follows from the $\omega_{\textrm{FS}}$- and $\omega_{0}$-tameness of $J_{\textrm{st}}$. 
This also shows that $J$ is $\widetilde{\omega}_{a,b}$-tame. 
Finally, we check the $J$-invariance of $\widetilde{\omega}_{a,b}$. 
The fact that $J_{\textrm{st}}$'s on $\CP^1$ and $\CP^2$ are compatible with $\omega_{0}$ and $\omega_{\textrm{FS}}$ proves that 
\begin{eqnarray*}
	\widetilde{\omega}_{a,b}(Ju, Jv) & = & (a-b)\tilde{p}^{*}\omega_0(Ju,Jv)+a\pi^{*}\omega_{\textrm{FS}}(Ju,Jv) \\ 
	& = &  (a-b)\omega_0(J_{\textrm{st}} (\tilde{p}_*u), J_{\textrm{st}}(\tilde{p}_*v))+a\omega_{\textrm{FS}}(J_{\textrm{st}}(\pi_*u),J_{\textrm{st}}(\pi_*v)) \\
	& = &  (a-b)\omega_0(\tilde{p}_*u, \tilde{p}_*v)+a\omega_{\textrm{FS}}(\pi_*u, \pi_*v) \\
	& = &  \widetilde{\omega}_{a,b}(u,v),
\end{eqnarray*}	
which completes the proof.
\end{proof}

We would like to explain the reason we have defined $\widetilde{\omega}_{a,b}$ as in (\ref{eqn: omega_tilde}), not as $a\pi^{*}\omega_{\textrm{FS}}+b\tilde{p}^{*}\omega_0$ for example. 
Observe that the symplectic form $\omega_{a,b}$ on $S^2 \times S^2$ evaluates to $a$ on the homology class $[S^2 \times \{\textrm{pt}\}]$ and to $b$ on the homology class $[\{\textrm{pt}\} \times S^2]$. 
Those homology classes can be seen as a section class and a fiber class of the $S^2$-bundle $S^2 \times S^2 \rightarrow S^2$, respectively. 
Similar evaluations hold for $\widetilde{\omega}_{a,b}$: 
Let $H$ be the proper transform of a generic hyperplane in $\CP^2$ not containing $x_0$, 
and also let $E$ be the exceptional sphere, i.e., $E=\pi^{-1}(x_0)$. 
Then $[H]$ and $[H]-[E]$ give rise to a $(+1)$-section class and a fiber class of the bundle $\tilde{p} \colon \CP^2 \# \overline{\CP}^2 \rightarrow \CP^1$, respectively. 
It follows from a straightforward calculation that 
$$
	\widetilde{\omega}_{a,b}([H])=a, \quad \widetilde{\omega}_{a,b}([H]-[E])=b.
$$

\subsubsection{The first Chern classes of the tangent bundles}

For future use, here we compute the first Chern classes $c_{1}(T(S^2 \times S^2))$ and $c_{1}(T(S^2 \widetilde{\times} S^2))$. 
Before the computation, recall the \textit{adjunction formula}: 
Let $X$ be a smooth complex surface and $C$ a smooth complex curve on $X$. 
Then, we have the following formula: 
\begin{equation}\label{eqn: adjuction}
	c_{1}(TX)([C]) = c_{1}(TC)([C])+c_{1}(N_{C/X})([C])=\chi(C)+[C]^{2},
\end{equation}
where $\chi(C)$ denotes the Euler characteristic of $C$. 

\begin{lemma}\label{lem: Chern class}
Let $T(S^2 \times S^2)$ (resp. $T(S^2 \widetilde{\times} S^2)$) be the tangent bundle of $S^2 \times S^2$ (resp. $S^2 \widetilde{\times} S^2$) with the standard complex structure. 
Then the following holds. 
\begin{enumerate}
\item\label{c_1 of S2xS2} The first Chern class of $T(S^2 \times S^2)$ is given by
$$
	c_{1}(T(S^2 \times S^2))=2\PD[S^2 \times \{ \mathrm{pt}\}]+2\PD[\{ \mathrm{pt}\} \times S^2].
$$
\item\label{c_1 of S2tilde{x}S2} The first Chern class of $T(S^2 \widetilde{\times} S^2)$ is given by
$$
	c_{1}(T(S^2 \widetilde{\times} S^2))=3\PD[H]-\PD[E].
$$
\end{enumerate}
\end{lemma}

\begin{proof}
We prove (\ref{c_1 of S2xS2}). 
For simplicity, let $A=[S^2 \times \{ \mathrm{pt}\}]$ and $B=[\{ \mathrm{pt}\} \times S^2]$.
Suppose that $c_{1}(T(S^2 \times S^2))= p\PD (A)+ q \PD (B)$ for some $p,q \in \Z$. 
We have 
$$
	c_{1}(T(S^2 \times S^2))(A) = q \ \ \textrm{and} \ \ c_{1}(T(S^2 \times S^2))(B) = p.
$$
By the adjunction formula (\ref{eqn: adjuction}), 
$$
	q=c_{1}(T(S^2 \times S^2))(A) = 2+0 = 2 \ \ \textrm{and} \ \ p=c_{1}(T(S^2 \times S^2))(B)=2+0 =2. 
$$
Thus, we have $c_{1}(T(S^2 \times S^2))=2\PD(A)+2\PD(B)$. 

We prove (\ref{c_1 of S2tilde{x}S2}). 
Suppose that $c_{1}(T(S^2 \widetilde{\times} S^2))= r\PD[E]+ s \PD[H]$ for some $r,s \in \Z$. 
Then, by the fact that $[H]^2=1$, $[E]^2=-1$ and $[E] \cdot [H]=0$, we have 
$$
	c_{1}(T(S^2 \widetilde{\times} S^2))([H]) = s \ \ \textrm{and} \ \ c_{1}(T(S^2 \widetilde{\times} S^2))([E]) = -r.
$$
As $H$ and $E$ are complex curves on $S^2 \widetilde{\times} S^2 (\approx \CP^2 \# \overline{\CP}^2)$, the formula (\ref{eqn: adjuction}) shows that 
$$
	s=c_{1}(T(S^2 \widetilde{\times} S^2))([H]) = 2+1=3 \ \ \textrm{and} \ \ -r=c_{1}(T(S^2 \widetilde{\times} S^2))([E])=2+(-1)=1. 
$$
Hence, we obtain $c_{1}(T(S^2 \widetilde{\times} S^2))=3\PD[H]-\PD[E]$. 
\end{proof}

\subsection{Embeddings of rational ruled surfaces}\label{section: embedding}

In this subsection, we will see that $(S^{2} \times S^{2}, \omega_{a,1})$ and 
$(S^{2} \widetilde{\times} S^{2}, \widetilde{\omega}_{a,1})$ can be embedded into 
closed symplectic manifolds as symplectic hyperplane sections when $a \geq 1$ for the former and $a \geq 2$ for the latter. 
Our embeddings will derive from complex geometry. 
Notice that according to Convention \ref{convention}, 
symplectic hyperplane sections are assumed to be of degree $1$ unless otherwise stated. 

%

\subsubsection{Case of $\omega_{1,1}$} First, let us consider the symplectic manifold $(S^{2} \times S^{2}, \omega_{1,1})$. 
Let $\omega_{\text{FS}}$ be the Fubini--Study form on $\CP^{4}$. 
Consider the projective hypersurface $M \subset \CP^{4}$ defined by 
$$
	M \coloneqq 
	\{ (z_{0}: z_{1}: z_{2}: z_{3}: z_{4})\in \CP^{4} \mid z_{0}^{2}+z_{1}^{2}+z_{2}^{2}+z_{3}^{2}+z_{4}^{2}=0 \},
$$
and we set $\Sigma = M \cap \{z_{4}=0 \}$.
Let $\omega$ be the symplectic form on $M$ defined to be the restriction of $\omega_{\text{FS}}$ to $M$. 
Let $\iota \colon \CP^1 \times \CP^1 \rightarrow M$ be the embedding given by 
$$
	([x:y], [z:w]) \mapsto [xz+ixw: i(xz+ixw): yz+iyw: i(yz+iyw):0] 
$$
whose image coincides with $\Sigma$. 
The cohomology class $[\iota^{*} \omega]$ evaluates to $1$ on the both of $[\CP^1 \times \{ \text{pt} \}]$ and $[\{\text{pt}\} \times \CP^1]$, which implies that 
identifying $\CP^1$ with $S^2$, $[\iota^{*}\omega]$ is cohomologous to $[\omega_{1,1}]$. 
Hence in view of \cite[Theorem 1.1]{LMc}, we conclude that $(S^2 \times S^2, \iota^{*}\omega)$ is 
symplectomorphic to $(S^{2} \times S^{2}, \omega_{1,1})$.

\begin{proposition}\label{prop: quadric}
The symplectic manifold $(\Sigma, \iota^{*}\omega)$ is a symplectic hyperplane section
on $(M, \omega)$. 
In particular, $(S^{2} \times S^{2}, \omega_{1,1})$ is embedded into $(M,\omega)$ 
as a symplectic hyperplane section. 
\end{proposition}

\begin{proof}
By definition, $(\Sigma, \iota^{*}\omega)$ is a symplectic submanifold of $(M, \omega)$. 
Since $\Sigma$ is a hyperplane section of $M$, we have $[\omega]=\textrm{PD}[\Sigma]$. 
Furthermore, since $\Sigma$ is an ample divisor on $(M, \omega)$, it follows that 
the complement $M \setminus \Sigma$ is a Stein manifold. 
This completes the proof. 
\end{proof}

\subsubsection{Case of $\omega_{a,1}$ and $\widetilde{\omega}_{a,1}$ with $a \geq 2$}
In this case, we embed the rational ruled surfaces into certain projective bundles. 
The following construction of embeddings is due to Biran and Cieliebak \cite[Section 2.4]{BC}.
Set $F_{m} \coloneqq \O(-m) \oplus \O \oplus \O \rightarrow \CP^{1}$, 
where $m \in \Z_{ \geq 0}$.  
Let $\P(F_{m})$ denote the projectivization of the vector bundle $F_{m}$ and 
$\pi \colon \P(F_{m}) \rightarrow \CP^{1}$ denote the bundle projection. 
Take two holomorphic sections $\sigma_{i}$ ($i=1,2$) of the hyperplane line bundle 
$\O_{\P(F_{m})}(1) \rightarrow \P(F_{m})$ induced by the projection 
$F_{m} \rightarrow \O$ from $F_{m}$ to the $i$-th trivial summand. 
We also take holomorphic sections $s_{1}$ and $s_{2}$ of the line bundle 
$\O(n) \rightarrow \CP^{1}$ ($n>0$) such that 
$s_{1}$ is transverse to the zero-section, and $s_{2}$ is transverse to the intersection of 
the zero-section and the zero set $Z(s_{1})$ of $s_{1}$, where we identify the base space $\CP^{1}$ with 
the zero-section.  
Define the smooth hypersurface $\Sigma_{m,n}$ of $\P(F_{m})$ by 
\begin{equation}\label{def: Sigma}
	\Sigma_{m,n}= Z(\sigma_{1} \otimes \pi^{*}s_{1}+\epsilon \sigma_{2}\otimes \pi^{*}s_{2}) 
\end{equation}
for a small $\epsilon >0$.

\begin{example}
We demonstrate the above construction, taking $F_{0}=\O \oplus \O \oplus \O$ as the simplest example. 
For $F_0$, the projective bundle $\P(F_{0}) \rightarrow \CP^1$ is nothing but the trivial bundle $\CP^1 \times \CP^2 \rightarrow \CP^1$. 
Letting $((z_0:z_1), (w_0:w_1:w_2))$ be homogeneous coordinates of $\CP^1 \times \CP^2$, we can regard the sections $\sigma_{1}$ and $\sigma_{2}$ of $\O_{\CP^1 \times \CP^2}(1) \rightarrow \CP^1 \times \CP^2$ as polynomials $w_1$ and $w_2$, respectively. 
Now choose sections $s_1, s_2 \colon \CP^1 \rightarrow \O(n)$ to be $z_0^n$ and $z_1^n$. 
Then by (\ref{def: Sigma}), we have 
$$
	\Sigma_{0,n} =\{((z_0:z_1), (w_0:w_1:w_2)) \in \CP^1 \times \CP^2 \mid w_1z_0^n+\epsilon w_2z_1^n=0\},
$$
which is the well-known description of rational ruled surfaces embedded in $\CP^1 \times \CP^2$ given by Hirzebruch. 
\end{example}

\begin{lemma}\label{lem: hypersurface}
The hypersurface $\Sigma_{m,n}$ is diffeomorphic to 
$S^{2} \times S^{2}$ if $m-n$ is even; otherwise, it is diffeomorphic to $S^{2} \widetilde{\times} S^{2}$. 
\end{lemma}

\begin{proof}
Since $\pi$ gives a $\CP^1$-bundle structure to $\Sigma_{m,n}$, it is a rational ruled surface.
Its diffeomorphism type is determined by the parity of the self-intersection number of a holomorphic section of the fibration $\Sigma_{m,n} \rightarrow \CP^{1}$ (e.g. see \cite[Theorem 3.4.8]{GS}). 
Hence it suffices to compute the one of 
$D_{\infty} \coloneqq \P(\O(-m) \oplus 0 \oplus 0)$ in $\Sigma_{m,n}$; this agrees with $c_{1}(N_{D_{\infty}/\Sigma_{m,n}})([D_{\infty}])$. 
The splitting $N_{D_{\infty}/ \P(F_{m})} \cong N_{D_{\infty}/\Sigma_{m,n}} \oplus N_{\Sigma_{m,n}/\P(F_{m})}|_{D_{\infty}}$ 
yields 
\begin{equation}
	c_{1}(N_{D_{\infty}/\Sigma_{m,n}})=c_{1}(N_{D_{\infty}/\P(F_{m})})-c_{1}(N_{\Sigma_{m,n}/\P(F_{m})}|_{D_{\infty}}).
\end{equation}
Thus, we shall focus on 
the right-hand side in the rest of the proof. 

Before the computation, 
we introduce some notations. 
Let $F$ be a general line in a fiber of $\P(F_{m})$ and 
$D_{0}$ the curve $\P(0 \oplus 0 \oplus \O)$.  
Set $\alpha = \pi^{*} c_{1}(\O(1))$ and $\beta =c_{1}(\O_{\P(F_{m})}(1)) \in H^{2}(\P(F_{m}); \Z)$. 
Then we have
\begin{eqnarray*}
	\alpha([F])=0, & \alpha([D_{0}])=1, &   \alpha([D_{\infty}])=1, \\
	\beta([F])=1,  & \beta([D_{0}])=0,  &  \beta([D_{\infty}])=m. 
\end{eqnarray*}
Here, the first three equalities follow from the fact that $\pi|_{F}$ is constant and $\pi|_{D_i}$ ($i=0,\infty$) are biholomorphisms, and the next three equalities follow from 
$$
	\O_{\P(F_m)}|_{F} \cong \O_{\CP^1}(1), \quad \O_{\P(F_m)}|_{D_0} \cong \O, \quad \O_{\P(F_m)}|_{D_{\infty}} \cong \O(m).
$$
Those evaluations lead to $[D_{\infty}]=[D_{0}]+m[F]$.

Let us compute $c_{1}(N_{D_{\infty}/\P(F_{m})})([D_{\infty}])$. 
By definition, 
$$
Z(\sigma_{1})=\P(\O(-m) \oplus 0 \oplus \O) \textrm{ and } Z(\sigma_{1}) \cap Z(\sigma_{2})=D_{\infty}, 
$$
and hence coupling the adjunction formula with the fact that $\O_{\P(F_{m})}(1)|_{D_{\infty}} \cong \pi^{*}\O(m)|_{D_{\infty}}$, we have  
\begin{eqnarray*}
N_{D_{\infty}/ \P(F_{m})} & \cong & 
N_{D_{\infty}/\P(\O(-m) \oplus 0 \oplus \O)} \oplus N_{\P(\O(-m) \oplus 0 \oplus \O)/ \P(F_{m})}|_{D_{\infty}} \\
& \cong & 
\O_{\P(\O(-m) \oplus 0 \oplus \O)}(1)|_{D_{\infty}} \oplus \O_{\P(F_{m})}(1)|_{D_{\infty}} \\ 
& \cong & 
(\pi^{*}\O(m) \oplus \pi^{*}\O(m))|_{D_{\infty}}.  
\end{eqnarray*}
Thus, 
\begin{equation}\label{eqn: normal 1}
	c_{1}(N_{D_{\infty}/\P(F_{m})})([D_{\infty}])=2m. 
\end{equation}

Next, compute $c_{1}(N_{\Sigma_{m,n}/\P(F_{m})}|_{D_{\infty}})([D_{\infty}])$. 
By the adjunction formula, we have 
$$N_{\Sigma_{m,n}/\P(F_{m})} \cong (\O_{\P(F_{m})}(1) \otimes \pi^{*} \O(n))|_{\Sigma_{m,n}}.$$ 
Thus, $c_{1}(N_{\Sigma_{m,n}/\P(F_{m})})=\iota^{*}(n\alpha+\beta')$ and 
\begin{equation}\label{eqn: normal 2}
	c_{1}(N_{\Sigma_{m,n}/\P(F_{m})}|_{D_{\infty}})([D_{\infty}])=n+m,  
\end{equation}
where $\iota \colon \Sigma_{n,m} \hookrightarrow \P(E_{m})$ denotes the inclusion.
Combining (\ref{eqn: normal 1}) with (\ref{eqn: normal 2}), we conclude that 
$$
	c_{1}(N_{D_{\infty}/ \P(F_{m})})[D_{\infty}]=2m-(n+m)=m-n. 
$$
Since $[F]^{2}=0$, the parity of $[D_{\infty}]^{2}=m-n$ determines 
the diffeomorphism type of $\Sigma_{m,n}$. 
\end{proof}

Now we discuss symplectic aspects of $\P(F_{m})$ and $\Sigma_{m,n}$. 
Thanks to semi-positivity of $\O(m)$ and $\O(n)$, 
the bundle $\O_{\P(F_{m})}(1) \otimes \pi^{*}\O(n) \rightarrow \P(F_{m})$ is positive, and 
$\P(F_{m})$ admits a K\"{a}hler form $\Omega_{m}$ with 
the cohomology class $[\Omega_{m}]$ equal to $n\alpha+\beta$ (see \cite[Section 2.4]{BC}); 
of course $\Omega_{m}$ depends on $n$ 
although we omit this from the notation for the sake of simplicity. 
Let $\eta_{m,n}$ denote the restriction of $\Omega_{m}$ to $\Sigma_{m,n}$.

\begin{lemma}\label{lemma: ruled}
The symplectic manifold $(\Sigma_{m,n}, \eta_{m,n})$ is symplectomorphic to one of the following: 
\begin{enumerate}
\item $(S^{2}\times S^{2}, \omega_{2m-3\ell, 1})$ if $m-n=2\ell$ is even; 
\item $(S^{2} \widetilde{\times} S^{2}, \widetilde{\omega}_{2m-3\ell-1,1})$ if $m-n=2\ell+1$ is odd.
\end{enumerate}
\end{lemma}

\begin{proof}
The following argument is similar to \cite[Proposition 2.2]{Bor}. 
As we have already seen, $\Sigma_{m,n}$ is a rational ruled surface. 
Hence, by a result in \cite[Theorem 1.1]{LMc} combined with Lemma \ref{lem: hypersurface}, 
the symplectomorphism type of $(\Sigma_{m,n}, \eta_{m,n})$ is determined by the parity of $m-n$ and 
the cohomology class of $[\eta_{m,n}] \in H^{2}(\Sigma_{m,n}; \Z)$. 
Recall from the proof of Lemma \ref{lem: hypersurface} that 
the homology classes $[D_{\infty}]$ and $[F]$ form a basis for $H_{2}(\Sigma_{m,n}; \Z)$, and 
their intersections are given by
$[D_{\infty}]^2=m-n, [F]^2=0$ and $[D_{\infty}] \cdot [F]=1$. 
We also have $\eta_{m,n}([D_{\infty}])=m+n $ and $\eta_{m,n}([F])=1$.

When $m-n=2\ell$ is even,
$(\Sigma_{m,n}, \eta_{m,n})$ is symplectomorphic to $(S^{2} \times S^{2}, \omega_{a,b})$ for some $a$ and $b$. 
The homology classes of the spheres of the latter space are identified with 
$$
	[S^{2} \times \{\text{pt}\}]= [D_{\infty}]-\ell[F] \ \ \textrm{ and  }\ \ 
	[\{\text{pt}\} \times S^{2}]= [F].
$$
This shows that $\omega_{a,b}([S^{2} \times \{\text{pt}\}])=2m-3\ell$ and 
$\omega_{a,b}([\{\text{pt}\} \times S^{2}])=1$. 
This proves the first assertion.

When $m-n=2\ell+1$ is odd,  
$(\Sigma_{m,n}, \eta_{m,n})$ is symplectomorphic to  $(S^{2} \widetilde{\times} S^{2}, \widetilde{\omega}_{a,b})$ for some $a$ and $b$. 
The homology classes $[H]$ and $[H]-[E]$ defined after Lemma \ref{lemma: Kahler}
are identified with 
$$
	[H]= [D_{\infty}]-\ell[F] \ \ \textrm{ and  }\ \ 
	[H]-[E]= [F], 
$$
and this shows that $\omega_{a,b}([H])=2m-3\ell-1$ and $\omega_{a,b}([H]-[E])=1$. 
This proves the second assertion.
\end{proof}


\begin{proof}[Proof of Theorem \ref{thm: embeddability} (\ref{item: embedding})]

In view of Proposition \ref{prop: quadric}, it suffices to deal with the case when $a \geq 2$.
It follows from Lemma \ref{lemma: ruled} that 
for any $a \geq 2$, there exist $m$ and $n$ such that 
$(S^{2} \times S^{2}, \omega_{a,1})$ is symplectomorphic to 
$(\Sigma_{m,n}, \eta_{m,n})$; this also holds for 
$(S^{2} \widetilde{\times} S^{2}, \widetilde{\omega}_{a,1})$. 
Now, suppose that $(\Sigma_{m,n}, \eta_{m,n})$ is symplectomorphic to 
either 
$(S^{2} \times S^{2}, \omega_{a,1})$ or $(S^{2} \widetilde{\times} S^{2}, \widetilde{\omega}_{a,1})$. 
The hypersurface $\Sigma_{m,n}$ is the smooth ample divisor on $\P(F_{m})$ given by 
$$
	Z(\sigma_{1} \otimes \pi^{*}s_{1}+\epsilon\sigma_{2}\otimes \pi^{*}s_{2}), 
$$
and the K\"{a}hler form $\Omega_{m}$ satisfies 
$[\Omega_{m}]=c_{1}(\O_{\P(F_{m})}(1) \otimes \pi^{*}\O(n))$. This shows that 
$\Sigma_{m,n}$ is a symplectic hyperplane section  on $(\P(F_{m}), \Omega_{m})$. 
\end{proof}

\begin{remark}
One could construct projective bundles $\P(F_m)$ and $\Sigma_{m,n}$ purely symplectically. 
The reason nevertheless we stick to work in the complex category is to prove that the complements $\P(F_m) \setminus \nu(\Sigma_{m,n})$ are subcritically Stein domains, applying a result of Biran and Cieliebak; see Section \ref{section: fillable}.
\end{remark}

\subsection{Non-embeddability of rational ruled surfaces}\label{section: non-embeddability}


\subsubsection{Almost complex structures on symplectic manifolds}
For non-embeddability results, 
we first define the set of almost complex structures on a symplectic manifold satisfying certain conditions.


Let $(M, \omega)$ be a symplectic manifold and $A \in H_{2}(M; \Z)$ a spherical class. 
Given a decomposition $A=A_{1}+ \cdots + A_{N}$ with $N \geq 1$ consisting of non-trivial spherical classes $A_{i}$'s,
we define the set $\mathcal{J}_{\mathrm{reg}}(\{ A_{i}\})$ of $\omega$-compatible almost complex structures $J$ on $M$ for which every simple $J$-holomorphic sphere representing a factor of $A$ is Fredholm regular. 
Considering all possible decompositions $\{ A_{i}\}$ of $A$, set 
$$ 
	{\mathcal{J}}_{\mathrm{reg}}(A) \coloneqq \bigcap_{\{A_{i}\}}\mathcal{J}_{\mathrm{reg}}(\{ A_{i}\}).
$$
It is known that $\mathcal{J}_{\mathrm{reg}}(A)$ is residual in the space of $\omega$-compatible almost complex structures on $M$ 
with respect to the $C^{\infty}$-topology 
(cf. \cite[Section 6.2]{MSbook}).

%

\subsubsection{Proof of non-embeddability}

Let $(S^{2}\times S^{2}, \omega_{a,2})$ and 
$(S^{2} \widetilde{\times} S^{2}, \widetilde{\omega}_{a,2})$ be the trivial and nontrivial $S^2$-bundles over $S^2$ endowed with the symplectic structures defined by (\ref{eqn: omega}) and (\ref{eqn: omega_tilde}), respectively. 
Let $F$ denote a fiber of the second projection $\mathrm{pr}_2 \colon S^2 \times S^2 \rightarrow S^2$ for $S^{2}\times S^{2}$ and to be one of the fibration $\tilde{p} \colon S^2 \widetilde{\times} S^2 \rightarrow S^2$ given in Section \ref{sec: Symplectic structures on S2tilde{x}S2} for $S^2 \widetilde{\times} S^2$. 
The following is a key lemma for non-embeddability results. 

\begin{lemma}\label{lemma: pi_1}
Let $\Sigma$ be a symplectic hyperplane section on a closed integral symplectic $6$-manifold $(M, \omega)$. 
Suppose that $(\Sigma, \omega|_{ \Sigma})$ is symplectomorphic to either 
$(S^2 \times S^2, \omega_{a,2})$ with $a>2$ or $(S^2 \widetilde{\times} S^2, \widetilde{\omega}_{a,2})$ with $a>3$, and 
the fiber class $B=\iota_{*}[F] \in H_{2}(M; \Z)$ is $J$-indecomposable for $J \in \mathcal{J}_{\mathrm{reg}}(B)$, 
where $\iota$ is a symplectic embedding of $(S^2 \times S^2, \omega_{a,2})$ or $(S^2 \widetilde{\times} S^2, \widetilde{\omega}_{a,2})$ with the image equal to $\Sigma$. 
Then, $\pi_{1}(M \setminus \Sigma)$ is isomorphic to $\Z/a\Z$ for $(S^2 \times S^2, \omega_{a,2})$ and $\Z/(a-2)\Z$ for $(S^2 \widetilde{\times} S^2, \widetilde{\omega}_{a,2})$. 
\end{lemma}

In the lemma, 
a \textit{$J$-indecomposable} homology class $A \in H_{2}(M; \Z)$ is a spherical class 
admitting no decomposition $A=A_{1}+ \cdots +A_{N}$ of $A$ with $N \geq 2$ such that each $A_{i}$ can be represented by a nonconstant 
$J$-holomorphic sphere. 

The following lemma will be used in the proof of Lemma \ref{lemma: pi_1}. 
\begin{lemma}\label{lem: H2isoinclusion}
Let $\iota$ be a symplectic embedding 
as in Lemma \ref{lemma: pi_1}. If $a > 2$, the induced homomorphism $\iota_{*}\colon H_{2}(S^2 \times S^2; \Z) \rightarrow H_{2}(M; \Z)$ is an isomorphism; so is $\iota_{*}\colon H_{2}(S^2 \widetilde{\times} S^2; \Z) \rightarrow H_{2}(M; \Z)$ if $a>3$. 
\end{lemma}

\begin{proof}
We only deal with $(S^2 \times S^2, \omega_{a,2})$ in this proof. 
A similar argument proves the assertion for $(S^2 \widetilde{\times} S^2, \widetilde{\omega}_{a,2})$.

The surjectivity of $\iota_{*}$ directly follows from the long exact sequence in homology for the pair $(M, \iota(S^2 \times S^2))$. Indeed, since $M \setminus \iota(S^2 \times S^2)$ admits a Stein structure, we have $H_{2}(M, \iota(S^2 \times S^2); \Z) \cong H_{2}(W, \del W; \Z) \cong H^{4}(W; \Z)=0$, where $W$ is the complement of a tubular neighborhood of $\iota(S^2 \times S^2)$ in $M$, and the long exact sequence 
$$
	\cdots \rightarrow H_{2}(\iota(S^2 \times S^2); \Z) \xrightarrow{\iota_{*}} H_{2}(M; \Z) \rightarrow H_{2}(M, \iota(S^2 \times S^2); \Z)=0 \rightarrow \cdots
$$
tells us that $\iota_*$ is surjective.
To see injectivity, 
suppose that there exist integers $k_{1}, k_{2}\in \Z$ such that $\iota_{*}(k_{1}A+k_{2}B)=0$, 
where we set $A=[S^2 \times \{\mathrm{pt}\}]$ and $B=[\{\mathrm{pt}\} \times S^2]=[F]$ as in the proof of Lemma \ref{lem: Chern class}. 
By definition of symplectic hyperplane sections, we have 
\begin{equation}\label{eqn: Chern number}
c_{1}(N_{\Sigma/ M})(\iota_{*}A)=a \textrm{\ \ and \ } c_{1}(N_{\Sigma/ M})(\iota_{*}B)=2.
\end{equation}
We obtain 
\begin{equation}\label{eqn: normal}
0=c_{1}(N_{\Sigma/ M})(0)=c_{1}(N_{\Sigma/ M})(\iota_{*}(k_{1}A+k_{2}B))=ak_{1}+2k_{2}.
\end{equation}
Moreover, using Lemma \ref{lem: Chern class}, we have 
$$
	c_{1}(TM)(\iota_{*}A)=c_{1}(N_{\Sigma/ M})(\iota_{*}A)+c_{1}(T(S^2 \times S^2))(A)=a+2,
$$
and 
$$
	c_{1}(TM)(\iota_{*}B)=c_{1}(N_{\Sigma/ M})(\iota_{*}B)+c_{1}(T(S^2\times S^2))(B)=4. 
$$
We deduce
\begin{equation}\label{eqn: tangent}
	0=c_{1}(TM)(0)=c_{1}(TM)(k_{1}\iota_{*}A+k_{2}\iota_{*}B)=(a+2)k_{1}+4k_{2}. 
\end{equation}
Combining the equations (\ref{eqn: normal})  and  (\ref{eqn: tangent}), we conclude that 
$k_{1}=0$ and $k_{2}=0$ under the assumption $a>2$, and this completes the proof.
\end{proof}

Now we prove Lemma \ref{lemma: pi_1}.

\begin{notation}
In the proof below, we will consider fundamental groups of several topological spaces. 
To avoid confusion, let us denote a homotopy class of a loop $f \colon S^1=\R/\Z \rightarrow X$ based at a point $p $
in a topological space $X$  by $[f]_{X}$ or $[f]_{(X,p)}$. 
\end{notation}


\begin{proof}[Proof of Lemma \ref{lemma: pi_1}]
The following proof is in spirit contained in \cite[Theorem 1.1(ii)]{Hind}. 
We first prove the case $(S^2 \times S^2, \omega_{a,2})$, divided into five steps. 
After this, as the other case is similar to the former case, 
we will mainly explain the difference between them and give a sketch of the proof.

\textbf{Case $(S^2 \times S^2, \omega_{a,2})$.}

\textbf{Step1.} We define a moduli space of holomorphic spheres in $M$. 

Since the homology class $[\Sigma]$ is Poincar\'{e} dual to $[\omega_{a,2}]$, the normal bundle $N_{\Sigma/M}$ is diffeomorphic to the line bundle $\O(a,2) \coloneqq \mathrm{pr}_1^{*}\O(a) \otimes \mathrm{pr}_2^{*}\O(2) \rightarrow S^2 \times S^2$. 
Here, $\mathrm{pr}_1$ and $\mathrm{pr}_2$ denote the first and second projections. 
We can therefore choose an $\omega$-compatible almost complex structure $J_{0}$ on $M$ so that 
a tubular neighborhood of $\Sigma$ in $M$ is biholomorphic to one of the zero-section of the holomorphic line bundle $\O(a,2) \rightarrow S^2 \times S^2$. 
Recall that 
$c_{1}(N_{\Sigma/M})(B)=[\omega_{a,2}](B)=2$ and 
$c_{1}(N_{\iota(\{\mathrm{pt} \} \times S^2)/\Sigma})(B)=0$.
In light of \cite[Lemma 3.3.1]{MSbook}, 
these evaluations of the first Chern classes together with integrability of $J_{0}$ near $\Sigma$ show 
that every simple $J_{0}$-holomorphic sphere near $\Sigma$ homologous to $B$ is Fredholm regular. 

Define the (unparametrized) $J_{0}$-holomorphic spheres $Q_{0}$ and $Q_{\infty}$ in $\Sigma$ to be 
$\iota(S^2 \times \{y_0\})$ and $\iota(S^2 \times \{y_{\infty}\})$ with $y_0 \neq y_{\infty}$, respectively. 
Note that $Q_{0}$ and $Q_{\infty}$ are disjoint. 
Consider the moduli space $\M(B; J_{0};Q_{0} \times Q_{\infty})$ consisting of all $J_{0}$-holomorphic spheres 
$u \colon \CP^1 \rightarrow M$ and points $z_{0}, z_{\infty}$ of $\CP^{1}$ such that $[u(\CP^1)]=B \in H_{2}(M;\Z)$ and 
$u(z_{0}) \in Q_{0}$ and $u(z_{\infty}) \in Q_{\infty}$: 
$$
	\M(B; J_{0};Q_{0} \times Q_{\infty}) = \left\{ 
	\begin{gathered}
	(u,z_{0},z_{\infty}) \in W^{1,p}(\CP^{1}, M) \times \CP^{1} \times \CP^{1} \mid \\
	u \textrm{ is $J_{0}$-holomorphic, }[u(\CP^{1})]=B \textrm{ and } u(z_{i}) \in Q_{i} \textrm{ for } i=0,\infty
	\end{gathered}
	\right\}.
$$
For the sake of simplicity we abbreviate $\M \coloneqq \M(B;J_{0};Q_{0} \times Q_{\infty}).$

According to the above argument, $J_0$ is chosen to make a neighborhood $U$ of $\Sigma$ biholomorphic to one of the zero-section of $\O(a,2)$ and all simple $J_{0}$-holomorphic spheres $u \colon \CP^1 \rightarrow M$ in $\M(B; J_{0};Q_{0} \times Q_{\infty})$ with $u(\CP^1) \subset U$ Fredholm regular. 
Moreover, thanks to \cite[Theorem 4.4.3]{Wen}, we can actually perturb this almost complex structure outside $U$ in such a way that the resulting almost complex structure lies in $\mathcal{J}_{\mathrm{reg}}(B)$ and all simple $J_{0}$-holomorphic spheres in $\M(B; J_{0};Q_{0} \times Q_{\infty})$ are Fredholm regular; hence we take such an almost complex structure as $J_0$ from the begginning.
By Lemma \ref{lem: H2isoinclusion}, the homology class $B$ is primitive in $H_2(M; \Z)$, and all $J_{0}$-holomorphic spheres in the moduli space are simple. 
It follows that $\M$ is a smooth oriented manifold of dimension $10$. 

\textbf{Step 2.}
We show that the evaluation map 
$\mathrm{ev} \colon \M \times_{G} \CP^{1} \rightarrow M$ is proper and is of degree one, where  
$G=\mathrm{Aut}(\CP^{1})$ acts on $\M \times \CP^{1}$ by 
$$
	g \cdot ((u,z_{0},z_{\infty}),z) \coloneqq ((u \circ g^{-1}, g(z_{0}), g(z_{\infty})), g(z))
$$
for $(u,z_{0},z_{\infty}) \in \M$, $z \in \CP^{1}$ and $g \in G$. 

By the assumption that $B$ is $J_{0}$-indecomposable, the space $\M \times_{G} \CP^{1}$ is compact 
(cf. the proof of \cite[Lemma 7.1.8]{MSbook}), and hence 
$\mathrm{ev}$ is proper.

To compute the degree of $\mathrm{ev}$, we claim that every point $p_{0} \in \Sigma \setminus (Q_{0} \cup Q_{\infty})$ is a regular value of $\mathrm{ev}$, and that 
$\mathrm{ev}^{-1}(p_{0})$ consists of only one point. This implies $\deg (\mathrm{ev})=1$.  By positivity of intersections, the unique $J_{0}$-holomorphic sphere in $\Sigma$ 
passing through $p_{0}$ in the class $B$ intersects $Q_{0}$ and $Q_{\infty}$. 
Note that $$
	B \cdot [\Sigma] = \int_{B} \omega = \int_{F}\omega_{a,2}=2, 
$$
that is, the intersection number of $J_{0}$-holomorphic spheres in $M$ representing $B$ with $\Sigma$ equals $2$. 
It follows from positivity of intersections 
that 
any $J_{0}$-holomorphic sphere passing through $Q_{0}$, $Q_{\infty}$ and $p_{0}$ must lie in $\Sigma$, and hence 
it is unique and $\mathrm{ev}^{-1}(p_0)$ is a one-point set.

We next prove that $p_{0}$ is a regular value of $\mathrm{ev}$. It is easy to see that for a given vector $Y \in T_{p_{0}}u(\CP^{1}) \subset T_{p_{0}}M$, 
there exists a vector $X \in T_{[(u,0,\infty), 1]}(\M \times_{G} \CP^{1})$ such that $d_{[(u,0,\infty),1]}\mathrm{ev}(X)=Y$.
To deal with normal directions, consider the pull-back bundle of the normal bundle $N_{u(\CP^{1})/M}$ to $u(\CP^{1})$ in $M$ under the unique $J_{0}$-holomorphic 
sphere $u \colon \CP^{1} \rightarrow M$ with $u(0) \in Q_{0}$, $u(\infty)\in Q_{\infty}$ and $u(1)=p_{0}$. 
In $\Sigma$, the self-intersection number of the sphere $u(\CP^1)$ is $0$, so that its normal bundle in $\Sigma$ is trivial. 
The bundle $u^{*}N_{u(\CP^{1})/M} \rightarrow \CP^{1}$ is isomorphic to 
$$
	u^{*}N_{u(\CP^{1})/M} \cong u^{*}N_{u(\CP^{1})/\Sigma}\oplus u^{*}N_{\Sigma/M} \rightarrow \CP^1, 
$$
namely $\O \oplus \O(2) \rightarrow \CP^{1}$ as holomorphic vector bundles. 
Hence, for a given vector $Y$ of $(N_{u(\CP^{1})/ M})_{p_{0}}$, 
holomorphic sections of this normal bundle give a curve $(-\epsilon, \epsilon) \ni s \mapsto [(u_{s},0,\infty), 1] \in \M \times_{G} \CP^{1}$ 
with $u_{0}=u$ such that 
$$
	d_{[(u,0,\infty),1]}\mathrm{ev}\left( \left. \frac{d[(u_{s},0,\infty),1]}{ds} \right|_{s=0} \right) 
	= \left. \frac{du_{s}(1)}{ds} \right|_{s=0}=Y \in T_{p_{0}}M.
$$
We conclude that the differential of $\mathrm{ev}$ at $p_0$ is surjective, and hence $p_{0}$ is a regular value. 

\textbf{Step 3.}
We construct a map $\varphi \colon \pi_{1}(X) \rightarrow \pi_1(D_0(\hat \xi))$, where $X = M \setminus \Sigma$ and the space $D_0(\hat \xi)$ is defined below.
 
Let 
$\xi$ be the total space of the normal bundle $N_{\Sigma/M}$ to $\Sigma$ in $M$  
restricted to $Q_{0}$. 
Denoting by $\nu_{\Sigma}(Q_{0})$ a tubular neighborhood  
of $Q_{0}$ in $\Sigma$, we write 
$\hat{\xi}$ for the total space of the restricted normal bundle $N_{\Sigma/M}|_{\nu_{\Sigma}(Q_{0})}$. 
Equip $N_{\Sigma/M}$ with a bundle metric and 
consider the total space $D(\hat{\xi})$ of the disk bundle associated to $N_{\Sigma/M}|_{\nu_{\Sigma}(Q_{0})}$,
and $D(\xi) \subset D(\hat{\xi})$ is defined in the same manner.
We will identify $D(\hat{\xi})$ with a tubular neighborhood of $\nu_{\Sigma}(Q_{0})$ in $M$ and 
regard $\nu_{\Sigma}(Q_{0})$ as a submanifold of this tubular neighborhood (see Figure \ref{figure} 
for schematic pictures).  The space $D_{0}(\hat{\xi})$ is defined to be the complement of the zero-section in 
$D(\hat{\xi})$. 
Notice that $D_{0}(\hat{\xi})$ is homotopic to the total space of a circle bundle associated to $\hat{\xi} \rightarrow \nu_{\Sigma}(Q_0)$, which is moreover homotopic to that of a circle bundle associated to $\xi=N_{\Sigma/M}|_{Q_0} \rightarrow Q_0$. 
This line bundle is isomorphic to $\O(a) \rightarrow \CP^1$ since $c_1(\xi)=a$ by (\ref{eqn: Chern number}); hence the fundamental group of $D_0(\hat{\xi})$ is isomorphic to $\Z/a\Z$.

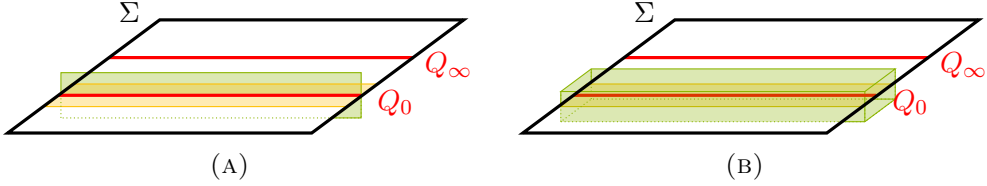
\begin{figure}[htb]
	\begin{subfigure}[b]{0.4\textwidth}
	\begin{tikzpicture}


\draw[draw=red, very thick] (1.33,1) -- (5.34,1);
\node at (5.8,0.9) {\color{red}$Q_{\infty}$} ;

\draw[draw=applegreen] (0.70,0.80) rectangle (4.65,0.50);
\fill[fill=applegreen, opacity=0.3] (0.70,0.80) rectangle (4.65,0.50);


\draw[draw=amber] (0.5,0.35) -- (4.48,0.35) -- (4.87,0.65) -- (0.84,0.65) -- cycle;
\fill[fill=amber, opacity=0.3] (0.5,0.35) -- (4.48,0.35) -- (4.87,0.65) -- (0.84,0.65) -- cycle;

\draw[draw=applegreen, densely dotted] (0.70,0.20) rectangle (4.65,0.50);

\draw[draw=applegreen] (4.65,0.20) -- (4.65,0.50) -- (4.25,0.20) -- cycle;
\fill[fill=applegreen, opacity = 0.3] (4.65,0.20) -- (4.65,0.50) -- (4.25,0.20);



\draw[draw=red, very thick] (0.67,0.5) -- (4.65,0.5) ;
\node at (5.1,0.4) {\color{red}$Q_0$};

\draw[draw=black, very thick]  (0,0) -- (4,0) -- (6,1.5) --  (2,1.5) --  cycle ;
\node at (1.6,1.6) {$\Sigma$};

\end{tikzpicture}	
\subcaption{}
\end{subfigure} \qquad
	\begin{subfigure}[b]{0.4\textwidth}
	\begin{tikzpicture}


\draw[draw=red, very thick] (1.33,1) -- (5.34,1) ;
\node at (5.8,0.9) {\color{red}$Q_{\infty}$} ;  ;


\draw[draw=amber] (0.5,0.35) -- (4.48,0.35) -- (4.87,0.65) -- (0.84,0.65) -- cycle;
\fill[fill=amber, opacity=0.3] (0.5,0.35) -- (4.48,0.35) -- (4.87,0.65) -- (0.84,0.65) -- cycle;

\draw[draw=applegreen, densely dotted] (0.5, 0.15) -- (4.5, 0.15) -- (4.9, 0.45) -- (0.9, 0.45) -- cycle;
\draw[draw=applegreen] (0.5, 0.55) -- (4.5, 0.55) -- (4.9, 0.85) -- (0.9, 0.85) -- cycle;
\draw[draw=applegreen] (0.5, 0.35) -- (0.5, 0.55);
\draw[draw=applegreen] (4.5, 0.15) -- (4.5, 0.55);
\draw[draw=applegreen] (4.9, 0.35) -- (4.9, 0.85);
\draw[draw=applegreen] (4.9, 0.45) -- (4.5, 0.15);
\draw[draw=applegreen] (4.5,0.15)  -- (4.2, 0.15);
\draw[draw=applegreen] (0.9,0.85)  -- (0.9, 0.65);
\draw[draw=applegreen] (4.9,0.45)  -- (4.6, 0.45);
\draw[draw=applegreen, densely dotted] (0.5, 0.15) -- (0.5, 0.55);
\draw[draw=applegreen, densely dotted] (4.5, 0.15) -- (4.5, 0.55);
\draw[draw=applegreen, densely dotted] (4.9, 0.45) -- (4.9, 0.85);
\draw[draw=applegreen, densely dotted] (0.9, 0.45) -- (0.9, 0.85);

\fill[fill=applegreen, opacity=0.3] (0.5, 0.15) -- (0.5, 0.35) -- (4.5, 0.35) -- (4.5, 0.15);

\draw[draw=red, very thick] (0.67,0.5) -- (4.65,0.5);
\node at (5.1,0.4) {\color{red}$Q_0$};

\fill[fill=applegreen, opacity=0.3] (0.5, 0.55) -- (4.5, 0.55) -- (4.9, 0.85) -- (0.9, 0.85) -- cycle;
\fill[fill=applegreen, opacity=0.3] (0.5, 0.55) -- (0.5, 0.35) -- (4.5, 0.35) -- (4.5, 0.55) -- cycle;
\fill[fill=applegreen, opacity=0.3] (4.5, 0.15) -- (4.5, 0.55) -- (4.9, 0.85) -- (4.9, 0.45) -- cycle;



\draw[draw=black, very thick]  (0,0) -- (4,0) -- (6,1.5) --  (2,1.5) --  cycle ;
\node at (1.6,1.6) {$\Sigma$};

\end{tikzpicture}
\subcaption{}
	\end{subfigure}
	\caption{(A) The neighborhoods $\nu_{M}(\Sigma)|_{Q_0}$ (green) and $\nu_{\Sigma}(Q_0)$ (yellow). (B) The neighborhood $D(\hat \xi)$ (green).}
	\label{figure}
\end{figure}


Now we construct a map $\varphi \colon \pi_{1}(X) \rightarrow \pi_{1}(D_{0}(\hat{\xi})) \cong \Z/a\Z$; 
this map shall provide the desired isomorphism for the assertion of the lemma.
Take a point $p \in X$ sufficiently close to a point on $\Sigma \setminus (Q_{0} \cup Q_{\infty})$ 
so that it is a regular value of the map $\mathrm{ev}$. 
We can perturb 
any loop $f \colon S^1 \rightarrow X$ based at $p$  
so that it is an embedding and transverse to the evaluation map $\mathrm{ev}$. 
Then, thanks to transversality, 
$\mathrm{ev}^{-1}(f(S^1))$ is a $1$-dimensional submanifold in $(\M \times_{G} \CP^{1}) \setminus \mathrm{ev}^{-1}(\Sigma)$ parametrized by $t \in S^1 = \R/\Z$. 
In fact, the component of $\mathrm{ev}^{-1}(f(S^1))$ mapping surjectively onto $f(S^1)$ is a circle since $\mathrm{ev}^{-1}(p)$ is a one-point set. 
Write 
$u_{t}^{f} \colon \CP^{1} \rightarrow M$ for 
a $J_{0}$-holomorphic sphere in this component such that 
$u_{t}^{f}(0) \in Q_{0}$, $u_{t}^{f}(\infty) \in Q_{\infty}$ and $u_{t}^{f}(1)=f(t)$, i.e.,  
$$
\mathrm{ev}([(u_{t}^{f},0,\infty), 1])=f(t).
$$ 
Note that $u_{0}^{f}=u_{1}^{f}$, and $u_{0}^{f}$ depends only on $p$ (not $f$); this 
allows us to set $u_{0} \coloneqq u_{0}^{f}$.
Take a sufficiently small real number $\epsilon_0 > 0$ such that the image $u_0 (D^2(\epsilon_0))$ of the closed disk $D^2(\epsilon_0) = \{ z \in \C \mid |z| \leq \epsilon_{0} \}$  is contained in $D(\hat \xi)$. We set 
$$p' \coloneqq u_0(\epsilon_0)$$ 
to be a base point for loops on $D_0(\hat \xi)$. Depending on the circle $u_t^f$ of $J_{0}$-holomorphic spheres, we can also take $
\epsilon >0$ such that $u_t^f(D^2(\epsilon))$ is contained in $D(\hat \xi)$ for all $t$. This provides the loop $g(f)$ in $D_0(\hat \xi)$ given by $g(f)(t) \coloneqq u_t^f(\epsilon)$. We may assume that $\epsilon < \epsilon_0$, and since $D_0(\hat \xi)$ is path-connected, we can regard $g(f)(t)$ as an element of $\pi_1(D_0(\hat \xi), p')$ up to base point change. Note that the choice of $\epsilon$ does not affect the homotopy class $[g(f)(t)]_{D_0 (\hat \xi)} \in \pi_1(D_0(\hat \xi), p')$. Assuming that $[g(f)(t)]_{D_0 (\hat \xi)}$ actually depends only on the homotopy class $[f]_X \in \pi_1(X, p)$, which will be shown in the next step,  we define a map $\varphi \colon \pi_{1}(X,p) \rightarrow \pi_{1}(D_{0}(\hat{\xi}), p')$ by
$$\varphi([f]_{X})=[g(f)]_{D_{0}(\hat{\xi})}.$$


\textbf{Step 4.}
We prove that $\varphi$ is a well-defined homomorphism.

To see its well-definedness as a map, 
take two homotopic smooth embeddings $f_{1}, f_{2} \colon S^1 \rightarrow X$ based at $p$ which are transverse to $\mathrm{ev}$. 
Since $\dim X=6 \geq 4=2(\dim S^1+1)$, there exists an isotopy $F \colon S^1 \times I \rightarrow X$ between $f_1$ and $f_2$ (see \cite[Exercise 8.1.10]{Hirsch}). 
We may assume that $F$ is transverse to $\mathrm{ev}$. 
By transversality, $\mathrm{ev}^{-1}(F(S^1 \times I))$ is an embedded submanifold of dimension $2$, 
which shows the lifts $\tilde{f}_{1}$ and $\tilde{f}_{2}$ in $(\M \times_{G}\CP^{1}) \setminus \mathrm{ev}^{-1}(\Sigma)$
of $f_{1}$ and $f_{2}$, respectively, are homologous. 
Hence, the corresponding loops $g(f_{1})$ and $g(f_{2})$ in $D_{0}(\hat{\xi})$ are homologous. 
As any two homologous loops in the latter space are homotopic, 
 $g(f_{1})$ and $g(f_{2})$ are homotopic, and $\varphi$ is a well-defined map.

To show that $\varphi$ is a homomorphism, 
take two smooth loops $f_{1}$ and $f_{2}$ based at $p$ which are embeddings and transverse to $\mathrm{ev}$. 
Let $f_{1,2}$ be a smoothing of the composition $f_{1} \cdot f_{2}$ meeting the same conditions as $f_i$. 
Using the transversality argument as in the previous paragraph, we find that respective lifts $\tilde{f}_{1} \cdot \tilde{f}_{2}$ and $\tilde{f}_{1,2}$ of $f_1 \cdot f_2$ and $f_{1,2}$ are homologous in $(\M \times_{G}\CP^{1}) \setminus \mathrm{ev}^{-1}(\Sigma)$, 
which implies that $g(f_{1} \cdot f_{2})$ and $g(f_{1, 2})$ are homologous in $D_{0}(\hat{\xi})$. 
Moreover, we find that the former lift $\tilde{f}_{1} \cdot \tilde{f}_{2}$ provides the family of holomorphic spheres defined by 
$$
	u_t=
	\begin{cases}
	u_{2t}^{f_1} & \textrm{for\ } t \in [0,1/2], \\
	u_{2t-1}^{f_2} & \textrm{for\ }t \in [1/2,1],
	\end{cases}
$$
and this yields a loop in $D_0(\hat{\xi})$ homologous to $g(f_1) \cdot g(f_2)$. 
Therefore, $g(f_1) \cdot g(f_2)$ is homologous to $g(f_1 \cdot f_2)$, and $\varphi([f_1]_{X}) \cdot \varphi([f_2]_{X})=\varphi([f_1]_{X} \cdot [f_2]_{X})$.

\textbf{Step 5.} 
We prove that $\varphi$ is an isomorphism, and hence $\pi_{1}(X)$ is isomorphic to $\Z/a\Z$. 

To show the injectivity of $\varphi$, suppose that
 $\varphi([f]_{X})=[g(f)]_{D_{0}(\hat{\xi})}$ is the identity element. 
Then, the map $S^1 \times [0,1] \rightarrow X$, $(t,s) \mapsto u_t^{f}((1-s)\epsilon_0+s)$ yields a free homotopy between $g(f)$ and $f$. 
As $g(f)$ is null-homotopic in $D_{0}(\hat{\xi})$, this proves that $[f]_{X}$ is trivial and $\varphi$ is injective.
 

Next, to see its surjectivity, 
consider the loop $f_{0}(t)$ in $X$ defined by $$f_{0}(t) = u_{0}(e^{2\pi i t}).$$ 
We can see that each point $u_{0}(e^{2\pi i t})$ on this loop is a regular value of $\mathrm{ev}$ as follows:  
Identify $\nu_{M}(\Sigma)$ with a neighborhood of the zero-section of 
$\O(a,2) \rightarrow S^2 \times S^2$. 
Then, one can take a holomorphic section $\tau \colon S^2 \times S^2 \rightarrow \O(a,2)$ such that $u_{0}$ lies on the image of $\tau$. 
In particular, $u_{0}(e^{2\pi i t})$ can be seen as a point of $\tau(S^2 \times S^2)$. 
By replacing $S^2 \times S^2$ by $\tau (S^2 \times S^2)$, 
an argument similar to that for regularity in Step 2 concludes that $u_{0}(e^{2 \pi i t })$ is a regular value of $\mathrm{ev}$.

Now we claim that the loop $g_{0}(t)$ in $D_{0}(\hat{\xi})$ defined by 
$$\label{def: g_0}
	g_{0}(t) \coloneqq g(f_{0})(t) = u_{0}(\epsilon_{0}e^{2\pi i t})
$$
gives rise to 
a generator of $\pi_{1}(D_{0}(\hat{\xi}), p')$. 
This loop bounds the disk $u_{0}(D^2(\epsilon_{0}))$, 
which intersects $\Sigma$ only at one point, namely $u_{0}(0) \in Q_{0}$. 
Let $S(\xi)$ and $S(\hat{\xi})$ denote the total spaces of the circle bundles associated to the restricted normal bundles 
$\xi \rightarrow Q_{0}$ and $\hat{\xi} \rightarrow \nu_{\Sigma}(Q_{0})$, 
respectively. 
The disk $u_{0}(D^2(\epsilon_{0}))$ may be assumed to give an element of not only 
$H_{2}(D(\hat{\xi}), S(\hat{\xi}); \Z)$ but also $H_{2}(D(\xi), S(\xi); \Z)$. 
Letting $I \colon H_{2}(D(\xi), S(\xi); \Z) \times H_{2}(D(\xi);\Z) \rightarrow 
H_{0}(D(\xi); \Z) \cong \Z$ be the intersection pairing, 
we have 
$$
	I([u_{0}(D^2(\epsilon_{0}))], [Q_{0}]) =1,
$$
where $[Q_{0}]$ is the class of the image of the zero-section of $D(\xi) \rightarrow Q_{0}$. 
Coupled with the fact that $H_{2}(D(\xi), S(\xi); \Z) \cong H^{2}(D(\xi); \Z) \cong \Z$, this shows that 
$[u_{0}(D^{2}(\epsilon_{0}))]$ is nontrivial, and especially it is a generator of 
$H_{2}(D(\xi), S(\xi); \Z)$. 
The group $H_{1}(D(\xi); \Z) \cong H_{1}(Q_{0}; \Z)$ is trivial, and hence the homomorphism 
$\del_{*} \colon H_{2}(D(\xi), S(\xi); \Z) \rightarrow H_{1}(S(\xi); \Z)$ appearing in the homology long exact sequence for 
the pair $(D(\xi), S(\xi))$ is surjective. 
Therefore, $\del_{*}$ maps  
$[u_{0}(D^2(\epsilon_{0}))]$ to a generator of $H_{1}(S(\xi); \Z) \cong \Z/a\Z$, and by definition we have 
$$\del_{*}([u_{0}(D^2(\epsilon_{0}))]) = [g_{0}].$$
Since the three groups $\pi_{1}(D_{0}(\hat{\xi}), p')$, $\pi_{1}(D_{0}(\xi))$ and $H_{1}(S(\xi); \Z)$ can be canonically identified, 
$[g_{0}]$ is a generator of $\pi_{1}(D_{0}(\hat{\xi}), p')$. This completes the proof of the case $(S^2 \times S^2, \omega_{a,2})$. 

\textbf{Case $(S^2 \widetilde{\times} S^2, \widetilde{\omega}_{a,2})$.}

Finally, we discuss the case of the nontrivial $S^2$-bundle. 
For use of complex structure, we regard $S^2 \widetilde{\times} S^2$ as $\CP^2 \# \overline{\CP}^2$ as in Section \ref{sec: Symplectic structures on S2tilde{x}S2}. 
Recall that $\tilde{p} \colon \CP^2 \# \overline{\CP}^2 \rightarrow \CP^1$ is 
the bundle projection and $F$ is its fiber. 
Also, $\pi \colon \CP^2 \# \overline{\CP}^2 \rightarrow \CP^2$ denotes the blow-up map, and $H$ and $E$ are the proper transform of a generic line in $\CP^2$ and the exceptional sphere, respectively. 
Note that $H$ and $E$ serve as disjoint holomorphic sections of $\tilde{p}$. 

Now we take the bundle $\widetilde{\O}(a,2)\coloneqq \tilde{p}^{*}\O(a-2) \otimes \pi^{*}\O(2) \rightarrow \CP^2 \# \overline{\CP}^2$ instead of $\O(a,2) \rightarrow S^2 \times S^2$ in the previous case. 
Note that $c_1(\widetilde{\O}(a,2))=c_1(N_{\Sigma/M})=[\widetilde{\omega}_{a,2}]$.
Let $\widetilde{J}_{0}$ be an element of $\mathcal{J}_{\mathrm{reg}}(B)$ which makes a neighborhood of $\Sigma$ biholomorphic to a neighborhood of the zero-section of $\widetilde{\O}(a,2)$ and all simple $\widetilde{J}_0$-holomorphic spheres in the fiber class $B$ Fredholm regular. 
Set 
$$
	\widetilde{Q}_{0} \coloneqq \iota(E)\quad \textrm{and} \quad \widetilde{Q}_{\infty} \coloneqq \iota(H)
$$
and  
$$
	\widetilde{\M}= \left\{ 
	\begin{gathered}
	(u,z_{0},z_{\infty}) \in W^{1,p}(\CP^{1}, M) \times \CP^{1} \times \CP^{1} \mid \\
	u \textrm{ is $\widetilde{J}_{0}$-holomorphic, }[u(\CP^{1})]=B \textrm{ and } u(z_{i}) \in \widetilde{Q}_{i} \textrm{ for } i=0,\infty
	\end{gathered}
	\right\}.
$$
Then, the space $\widetilde{\M} \times_{G} \CP^1$ is a compact oriented $6$-dimensional manifold under the assumption of the $\widetilde{J}_0$-indecomposability of $B$. 
Furthermore, the evaluation map $\mathrm{ev} \colon \widetilde{\M} \times_{G} \CP^1 \rightarrow M$ is of degree $1$ as in the previous case. 

What is left is the construction of an isomorphism $\pi_{1}(X) \rightarrow \Z/(a-2)\Z$. 
In this case, $\Z/(a-2)\Z$ comes from a space corresponding to $D_0(\hat{\xi})$, namely the complement of $\Sigma$ in a neighborhood of $\widetilde{Q}_{0}$: 
To see this, let $\widetilde{D}_{0}$ be this complement. 
A neighborhood of $\widetilde{Q}_0$ can be identified with a neighborhood of the zero section of $N_{\Sigma/M}|_{\nu_{\Sigma}(\widetilde{Q}_0)}$. 
Since 
$$
	c_1(N_{\Sigma/M})([\widetilde{Q}_{0}])=[\widetilde{\omega}_{a,2}]([E])=a-2,
$$
the bundle $N_{\Sigma/M}|_{\widetilde{Q}_0} \rightarrow \widetilde{Q}_0$ is isomorphic to a line bundle over the $2$-sphere with the minimal Chern number $a-2$. 
Hence, $\widetilde{D}_0$ is homotopic to the associated circle bundle and has the fundamental group isomorphic to $\Z/(a-2)\Z$. 
The similar argument to the previous case gives an isomorphism $\varphi \colon \pi_{1}(X) \rightarrow \pi_1(\widetilde{D}_{0}) \cong \Z/(a-2)\Z$.
\end{proof}

\begin{remark}
In the above proof, showing surjectivity is rather complicated. 
Here we would like to explain the reason for this. 
Given any loop $g \in D_0(\hat{\xi})$ based at $p'$ and transverse to the evaluation map $\mathrm{ev}$, similarly to Step 3, we get a family $u_{t}^{g} \colon \CP^1 \rightarrow M$ of holomorphic spheres such that $u_{t}^{g}(\epsilon)=g(t)$ and $u_{0}^{g}(1)=p$ for some $0 < \epsilon< 1$. 
One would expect that $\varphi([u_{t}^{g}(1)]_{X})=[g]_{D_0(\hat{\xi})}$, and surjectivity follows. 
However, notice that $u_{t}^{g}$ can escape from $D(\hat{\xi})$ along the interval $[0,\epsilon]$ in general.  
This fact makes us unable to prove that $\varphi([u_{t}^{g}(1)]_{X})=[g]_{D_0(\hat{\xi})}$. 
Thus, we employed the specific loop $g_0(t)=u_0(\epsilon_0 e^{2\pi it})$, in the proof, which bounds a disk in $D(\hat{\xi})$ containing the whole of $u_0([0,\epsilon_0])$.
\end{remark}

The $J$-indecomposability assumption in Lemma \ref{lemma: pi_1} is satisfied for most positive integers $a$:

\begin{lemma}\label{lemma: indecomposability}
Let $(M,\omega)$ be a closed integral symplectic $6$-manifold admitting a symplectic hyperplane section 
symplectomorphic to either $(S^2 \times S^2, \omega_{a,b})$ or $(S^2 \widetilde{\times} S^2, \widetilde{\omega}_{a,b})$. 
Then, the following hold: 
\begin{enumerate}
\item\label{item: trivial case} If $a \geq 5$ and $b=2$ for $\omega_{a,b}$, then the homology class $B$ corresponding to $[F]$ is 
$J$-indecomposable for $J \in \mathcal{J}_{\textrm{reg}}(B)$; 
\item\label{item: non-trivial case}  If $a \geq 7$ and $b=2$ for $\widetilde{\omega}_{a,b}$, then the homology class $B$ corresponding to $[F]$ is 
$J$-indecomposable for $J \in \mathcal{J}_{\textrm{reg}}(B)$. 
\end{enumerate}
\end{lemma}

\begin{proof}
We only prove (\ref{item: trivial case}). 
The proof of (\ref{item: non-trivial case}) is similar to that of (\ref{item: trivial case}). 

Let $\iota \colon S^2 \times S^2 \hookrightarrow M$ denote a symplectic embedding with $\iota(S^2 \times S^2)=\Sigma$. 
Set 
$$
	A = \iota_{*}[S^2 \times \{\mathrm{pt}\}] \textrm{ and } B = \iota_{*}[F]=\iota_{*}[\{ \mathrm{pt} \} \times S^2] \in H_{2}(M; \Z).
$$ 
Suppose on the contrary that $B$ is not $J$-indecomposable, that is, there exist homology classes $C_{1}, \ldots, C_{N} \in H_{2}(M; \Z)$ ($N \geq 2$)
such that $B=C_{1}+ \cdots + C_{N}$ and each $C_{i}$ is represented by a nonconstant $J$-holomorphic sphere $u_i \colon \CP^1 \rightarrow M$. 
Since $[\omega]$ has an integral lift, 
the minimal symplectic energy of nonconstant $J$-holomorphic spheres in $M$ must be $1$, and this implies that $N=2$ and the symplectic energy $E(u_i)\coloneqq \int_{\CP^1} u_i^{*}\omega $ of $u_i$ is $1$ for $i=1,2$. 
Moreover, by Lemma \ref{lem: H2isoinclusion}, 
we can represent each $C_{i}$ as a linear combination of $A$ and $B$; namely
$C_{i}=k_{i}A+\ell_{i}B$ for some $k_{i}, \ell_{i} \in \Z$ with 
\begin{equation} \label{eqns}
	k_{1}+k_{2}=0 \ \ \textrm{and\ \ } \ell_{1}+\ell_{2}=1.  
\end{equation}
Recall that $c_1(T(S^2\times S^2))=2\mathrm{PD}(A)+2\mathrm{PD}(B)$ given in Lemma \ref{lem: Chern class} (\ref{c_1 of S2xS2}). 
The expected dimension of the moduli space of (unparametrized) $J$-holomorphic spheres $u\colon \CP^{1} \rightarrow M$ with 
$[u(\CP^1)]= kA + \ell B$ is given by
$$
	\chi(\CP^{1}) \cdot \frac{\dim M}{2} + 2c_{1}(TM)([u(\CP^{1})])-\dim \mathrm{Aut}(\CP^{1}) =  2((a+2)k+4\ell),   
$$
where $\chi(\CP^{1})$ is the Euler characteristic of $\CP^{1}$.
Suppose that the symplectic energy of such $J$-holomorphic curves is equal to $1$, i.e., 
$$
	E(u) = \int_{kA+ \ell B} \iota^{*}\omega=ak+2\ell=1. 
$$
This assumption yields the fact that all $J$-holomorphic spheres in this moduli space are simple; 
As a consequence, the moduli space is a manifold. 
For this space to be nonempty, $k$ and $\ell$ need to satisfy the following inequality and equation: 
$$
	2((a+2)k+ 4\ell) \geq 0  \ \ \textrm{and\ \ } 
	ak+2\ell=1.
$$
From these we have 
\begin{equation}\label{ineq: moduli}
	k \leq \frac{2}{a-2} <1,  
\end{equation}
where the last inequality follows from the assumption that $a \geq 5$.

Let us return to our case. 
By the inequality (\ref{ineq: moduli}) combined with the equation for $k_{i}$ in (\ref{eqns}), we obtain 
$k_{1} = k_{2} = 0$. 
However, note that there is no $J$-holomorphic sphere $u \colon \CP^{1} \rightarrow M$ with $[u(\CP^{1})]=\ell_{i}B$ and $E(u)=1$. This leads to 
a contradiction.  
\end{proof}



We are ready to prove the second assertion in Theorem \ref{thm: embeddability}.

\begin{proof}[Proof of Theorem \ref{thm: embeddability} (\ref{item: non-embedding})]


In the following proof, we will use the same notations as 
in the proof of Lemma \ref{lemma: pi_1}. 
Also we will only discuss the case $(S^2 \times S^2, \omega_{a,2})$ ; the other case $(S^2 \widetilde{\times} S^2, \widetilde{\omega}_{a,2})$ is straightforward from the former case.
 
Suppose on the contrary that there exists a symplectic embedding 
$\iota\colon  (S^2 \times S^2, \omega_{a,2}) \rightarrow M$ into a closed integral symplectic $6$-manifold $(M, \omega)$ such that 
the image $\Sigma = \iota(S^2 \times S^2)$ is a symplectic hyperplane section on $(M, \omega)$. 
By Lemma \ref{lemma: indecomposability} (\ref{item: trivial case}), the homology class $B = \iota_{*}[F]$ is 
$J$-indecomposable for any $J \in \mathcal{J}_{\mathrm{reg}}(B)$, and $\Sigma$ and $B$ satisfy the assumption of Lemma \ref{lemma: pi_1}. 
Therefore the fundamental group $\pi_1(X)$ of the complement $X = M \setminus \Sigma$ is isomorphic to $\Z/a \Z$. 

Consider a sphere $R=\iota(F') \subset \Sigma$
passing through $u_{0}(0)$, where $F'$ is a fiber of $\mathrm{pr}_2 \colon S^2 \times S^2 \rightarrow S^2$.
Note that $[R]=B \in H_{2}(M; \Z)$. 
Let $\eta$ be the total space of the restricted normal bundle $N_{\Sigma/M}|_{R}$ and  
$\hat{\eta}$ the total space of the normal bundle $N_{\Sigma/M}|_{\nu_{\Sigma}(R)}$. 
Write $D(\eta)$, $S(\eta)$, $D(\hat{\eta})$ and $S(\hat{\eta})$ for the total spaces of the disk and circle bundles 
associated to $\eta$ and $\hat{\eta}$, respectively, 
with respect to some bundle metric on $N_{\Sigma/ M}$. 
Consider the loop $g_{0} \colon S^1 \rightarrow D_{0}(\hat{\xi})$ 
defined by 
$$
	g_{0}(t)=u_{0}(\epsilon_{0}e^{2\pi i t}).
$$ 
Taking $\epsilon_{0}$ sufficiently small if necessary, 
we may assume that the image $g_{0}(S^1)$ lies in $D_{0}(\hat{\xi}) \cap D_{0}(\hat{\eta})$, 
where $D_{0}(\hat{\eta})$ is the complement of the image of the zero-section in $D(\hat{\eta})$. 
Observe that $g_0(S^1)$ bounds a disk in $D(\hat{\eta})$, which intersects $Q_{0}$ only at one point as in the proof of surjectivity in Lemma \ref{lemma: pi_1}.
This shows that 
$g_{0}$ serves as a generator of the group 
$ \pi_{1}(D_{0}(\hat{\eta}), p') \cong \Z/2\Z$.

Write $j \colon D_{0}(\hat{\eta}) \hookrightarrow X$ for the inclusion map. 
Define the base change map $\phi_{h_{0}} \colon \pi_{1}(X, p') \rightarrow \pi_{1}(X, p)$ by 
$$
	\phi_{h_{0}}([f]_{(X,p')})=[h_{0} \cdot f \cdot h_{0}^{-1}]_{(X,p)}.
$$ 
Then, by definition we have $(\phi_{h_{0}} \circ j_{*})([g_{0}]_{D_{0}(\hat{\eta})})=[f_{0}]_{X}$, 
which is a generator 
of $\pi_{1}(X,p)$, and in particular it is a nontrivial element according to Lemma \ref{lemma: pi_1}. 
Hence, it follows that $\phi_{h_{0}} \circ j_{*} \colon  \pi_{1}(D_{0}(\hat{\eta}), p') \rightarrow \pi_{1}(X, p)$ is a 
nontrivial homomorphism. 
However, $\pi_{1}(D_{0}(\hat{\eta}), p') \cong \Z/2\Z \not\cong \Z/a\Z \cong \pi_{1}(X,p)$, and $a$ and $2$ are coprime. 
This contradicts the existence of the above nontrivial homomorphism. 
\end{proof}


\section{Fillability of Boothby--Wang bundles over rational ruled surfaces}\label{section: fillability}

%

\subsection{Boothby--Wang bundles over rational ruled surfaces}\label{section: BW}

Let $(M,\omega)$ be a closed integral symplectic manifold with a fixed integral lift $[\omega] \in H^2(M; \Z)$. 
There is a unique (up to isomorphism) principal $S^{1}$-bundle $p: P \rightarrow M$ with Euler class 
$e(P)=-[\omega]$. 
We can take a connection $1$-form $A \in \Omega^{1}(P)$ on $P$ with 
$dA=2\pi p^{*}\omega.$ 
Throughout this paper, we regard connection $1$-forms on principal $S^{1}$-bundles as usual $\R$-valued differential forms on them. 
By definition, the connection $1$-form $A$ serves as a contact form on $P$ whose Reeb orbits are fibers of $p$. 
We call the manifold $P$ the \textit{Boothby--Wang manifold} over $(M,\omega)$ 
and refer to the contact manifold $(P, \ker(A))$ as the 
\textit{Boothby--Wang contact manifold} over $(M,\omega)$.

Let $(P_{a,b}, \xi_{a,b})$ and 
$(\widetilde{P}_{a,b}, \widetilde{\xi}_{a,b})$ 
be the Boothby--Wang contact manifolds over the (integral) symplectic rational ruled surfaces
$(S^{2} \times S^{2}, \omega_{a,b})$ and $(S^{2} \widetilde{\times} S^{2}, \widetilde{\omega}_{a,b})$, respectively.

\begin{proposition}\label{prop: BW}
Suppose that $a$ and $b$ are coprime. 
Then, 
\begin{enumerate}
\item\label{item: BW1} $P_{a,b}$ is diffeomorphic to $S^{2} \times S^{3}$; 
$\widetilde{P}_{a,b}$ is diffeomorphic to $S^{2} \widetilde{\times} S^{3}$, where 
$S^{2} \widetilde{\times} S^{3}$ denotes the non-trivial $S^{3}$-bundle over $S^{2}$; 

\item\label{item: BW2} $c_{1}(\xi_{a,b})=(2a-2b)\gamma$ and $c_{1}(\widetilde{\xi}_{a,b})=(2a-3b)\widetilde{\gamma}$,
where $\gamma$ and $\widetilde{\gamma}$ are 
generators of $H^{2}(P_{a,b}; \Z)$ and $H^{2}(\widetilde{P}_{a,b}; \Z)$, respectively. 
\end{enumerate}
\end{proposition}

\begin{proof} 
In this proof, we only deal with $(P_{a,b}, \xi_{a,b})$. 
A similar argument also works for $(\widetilde{P}_{a,b}, \widetilde{\xi}_{a,b})$.

To prove (\ref{item: BW1}), we appeal to \cite[Proposition 8.2.1]{Gei} which states that a simply connected $5$-manifold $P$ with $w_2(P)=0$ and $H_2 (P;\Z)\cong \Z$ is diffeomorphic to $S^2 \times S^3$. 
We first claim that 
$H_{3}(P_{a,b}; \Z) \cong H^{2}(P_{a,b};\Z) \cong \Z$.
Applying the Gysin sequence for the bundle $\pi_{a,b} \colon P_{a,b} \rightarrow S^2 \times S^2$, we have 
\begin{align*}
	\cdots \rightarrow H^0(S^2 \times S^2; \Z) \xrightarrow{\smile e(P_{a,b})} H^2(S^2 \times S^2; \Z) \rightarrow H^2(P_{a,b}; \Z) \rightarrow 0 \rightarrow \cdots,
\end{align*}
which shows that the cohomology group $H^{2}(P_{a,b}; \Z)$ is isomorphic to 
$H^{2}(S^{2} \times S^{2}; \Z)/\Z \langle [\omega_{a,b}] \rangle$ since $e(P_{a,b})=-[\omega_{a,b}]$. 
To see that the latter group is isomorphic to $\Z$, 
consider the homomorphism defined by 
$$
	\Phi \colon H^{2}(S^{2} \times S^{2}; \Z) \cong \Z\langle \mathrm{pr}_1^{*}\omega_0 \rangle \oplus \Z\langle \mathrm{pr}_2^{*}\omega_0 \rangle \rightarrow \Z, \ \ 
	(r,s) \mapsto as-br,
$$
where $\mathrm{pr}_{i} \colon S^2 \times S^2 \rightarrow S^2$ denotes the $i$-th projection and $\omega_0$ is the standard area form on $S^2$ as in Section \ref{sec: Symplectic structures on S2xS2}.
The condition that $\Phi(r,s)=as-br=0$ is equivalent to $(r,s)=k(a,b)$ for some $k \in \Z$, which deduces that $\ker \Phi = \Z \langle [\omega_{a,b}] \rangle$. 
Moreover, since $a$ and $b$ are coprime, 
$\Phi$ is surjective. 
Hence, the first isomorphism theorem concludes the claim. 

Next, we prove that $P_{a,b}$ is simply connected. 
Together with the fact that $S^{2} \times S^{2}$ is simply connected, the homotopy exact sequence for the $S^1$-bundle $\pi_{a,b}$ implies that $\pi_{1}(P_{a,b})$ is abelian. 
Now in view of the isomorphism $\pi_{1}(P_{a,b}) \cong H_{1}(P_{a,b}; \Z) \cong H^{4}(P_{a,b}; \Z)$, it suffices to check that $H^{4}(P_{a,b};\Z)$ vanishes. 
To see this, consider the Gysin sequence for $\pi_{a,b}$ again:  
\begin{align*}
	\cdots \rightarrow H^2(S^2 \times S^2; \Z) \xrightarrow{\smile e(P_{a,b})} H^4(S^2 \times S^2; \Z) \rightarrow H^4(P_{a,b}; \Z) \rightarrow 0 \rightarrow \cdots,
\end{align*}
where the last trivial group in the sequence is $H^3(S^2 \times S^2; \Z)$. 
It turns out that $H^4(P_{a,b}; \Z) \cong H^4(S^2 \times S^2 ; \Z)/\mathrm{Im}(\smile e(P_{a,b}))$. 
The indivisibility of the Euler class $e(P_{a,b})$ implies the surjectivity of the map $\smile e(P_{a,b})$, and hence the group in question is trivial. 

The tangent bundle $TP_{a,b}$ splits into the horizontal and vertical directions as $TP_{a,b} \cong \pi_{a,b}^{*}T(S^2 \times S^2) \oplus \underline{\R}$. 
The second Stiefel--Whitney class $w_2(P_{a,b})$ of $P_{a,b}$ is given by 
$$
	w_2(P_{a,b}) =w_2(\pi_{a,b}^{*}T(S^2 \times S^2) \oplus \underline{\R})= \pi_{a,b}^{*}w_2(S^2 \times S^2)=0, 
$$ 
where the last equality follows from the fact that $S^2 \times S^2$ carries a spin structure. 

As $S^2\times S^2$ and $P_{a,b}$ are simply connected, the Hurewicz maps $\pi_{2}(S^2 \times S^2) \rightarrow H_2(S^2 \times S^2; \Z)$ and $\pi_{2}(P_{a,b}) \rightarrow H_2(P_{a,b};\Z)$ are isomorphisms. 
The homotopy exact sequence for the bundle $\pi_{a,b}$ reduces to the exact sequence of homology groups: 
$$
	 \cdots \rightarrow 0  \rightarrow  H_{2}(P_{a,b}; \Z) \xrightarrow{{\pi_{a,b}}_{*}} H_{2}(S^{2} \times S^{2}; \Z) \rightarrow \cdots. 
$$
It turns out that the induced map ${\pi_{a,b}}_{*} \colon H_{2}(P_{a,b}; \Z) \rightarrow H_{2}(S^{2} \times S^{2}; \Z) \cong \Z \oplus \Z$ is injective, which 
implies that $H_{2}(P_{a,b}; \Z) \cong H^{3}(P_{a,b}; \Z)$ is torsion free. 
Thus, the universal coefficient theorem combined with the above claim concludes that 
$$
	H_{2}(P_{a,b}; \Z) \cong \mathrm{Hom}(H^{2}(P_{a,b}; \Z), \Z) \oplus 
	\mathrm{Ext}(H^{3}(P_{a,b}; \Z), \Z) \cong \Z. 
$$
Now (\ref{item: BW1}) follows from \cite[Proposition 8.2.1]{Gei} as we wish.

In order to prove (\ref{item: BW2}), we first note that $\xi_{a,b}$ and $\pi_{a,b}^{*}T(S^2 \times S^2)$ are mutually isomorphic as complex vector bundles for a suitable complex structure on $\xi_{a,b}$. 
Hence,  
$$
	c_{1}(\xi_{a,b})=\pi_{a,b}^{*} c_{1}(T(S^{2} \times S^{2})). 
$$ 
Observe that the induced map $\pi_{a,b}^{*} \colon H^2(S^2 \times S^2; \Z) \rightarrow H^2(P_{a,b}; \Z) \cong \Z$ agrees with $\Phi$ up to sign. 
By Lemma \ref{lem: Chern class}, 
$$
	c_{1}(\xi_{a,b})=\Phi(2, 2)=(2a-2b)\gamma  
$$ 
with a generator $\gamma$ of $H^{2}(P_{a,b}; \Z)$. 
\end{proof}

\subsection{Stein fillable contact structures}\label{section: fillable}

Let us prove Theorem \ref{thm: fillability} (\ref{item: critically fillable}) and (\ref{item: subcritically fillable}). 

\begin{proof}[Proof of Theorem \ref{thm: fillability} (\ref{item: critically fillable}) and (\ref{item: subcritically fillable})]

Let $(M, \omega)$ be the complex quadric in $(\CP^{4}, \omega_{\text{FS}})$ and $\Sigma$ the hyperplane section in $M$ as in Section 
\ref{section: embedding}: 
$$
	M = \left\{ (z_{0}: z_{1}: z_{2}: z_{3}: z_{4})\in \CP^{4} \ \middle| \ z_{0}^{2}+z_{1}^{2}+z_{2}^{2}+z_{3}^{2}+z_{4}^{2}=0 \right\}, \ \ 
	\Sigma =M \cap \{ z_{4}=0\}. 
$$
By Proposition \ref{prop: quadric}, $(\Sigma, \omega|_{\Sigma})$ is a symplectic hyperplane section on $(M, \omega)$ 
and is symplectomorphic to $(S^{2} \times S^{2}, \omega_{1,1})$. 
Hence, the contact manifold $(P_{1,1}, \xi_{1,1})$ is obtained as the convex boundary of 
the complement of a symplectic tubular neighborhood of $(\Sigma, \omega|_{\Sigma})$ in $M$. 
Furthermore, $M \setminus \Sigma$ admits a Stein structure and 
is written as the affine hypersurface $X$ given by 
$$
	X=\left\{ (w_{0},w_{1}, w_{2}, w_{3})\in \C^{4} \ \middle| \ w_{0}^{2}+w_{1}^{2}+w_{2}^{2}+w_{3}^{2}+1=0 \right\}.
$$
This $X$ is diffeomorphic to the cotangent bundle $T^{*}S^{3}$. 
Indeed, an explicit diffeomorphism $X \rightarrow T^{*}S^3$ is given by 
$$
	X \rightarrow T^{*}S^3, \quad w \mapsto \left( \frac{\mathrm{Im}(w)}{|\mathrm{Im}(w)|},  \mathrm{Re}(w)|\mathrm{Im}(w)|\right),
$$
where we identify $T^{*}S^3$ with $\{(q,p) \in \R^4 \times \R^4 \mid \|q\|=1, q \cdot p=0\}$ (cf. \cite[Exercise 6.3.7]{McS}).
According to \cite[Theorem 1.2(1)]{BGZ}, 
all Stein fillings of a subcritically Stein fillable contact manifold have isomorphic 
homology groups. 
In particular, their middle-dimensional ones are trivial. 
Thus, $H_{3}(X; \Z) \cong H_3(T^*S^3;\Z)$ being non-trivial implies that 
$(P_{1,1}, \xi_{1,1})$ does not admit a subcritical Stein filling, 
which concludes that $\xi_{1,1}$ is critically Stein fillable. 

In the rest of the proof, we show the second assertion of $\xi_{a,1}$ with $a \geq 2$; 
the same argument also works for $\widetilde{\xi}_{a,1}$.
Following the proof of Theorem \ref{thm: embeddability} (\ref{item: embedding}), 
we have an identification 
$$
	(S^{2} \times S^{2}, \omega_{a,1}) \cong (\Sigma_{m,n}, \eta_{m,n}) 
$$
for some $m$ and $n$, and $(\Sigma_{m,n}, \eta_{m,n})$ is a symplectic hyperplane section on $(\P(F_{m}), \Omega_{m})$. 
Thus, the Boothby--Wang contact manifold $(P_{a,1}, \xi_{a,1})$ appears as the boundary of the complement of 
a symplectic tubular neighborhood 
of $\Sigma_{m,n}$ in $\P(F_{m})$. 
Employing the result of Biran and Cieliebak in \cite[Theorem 2.4.1]{BC}, 
this complement is in fact a subcritical Stein domain. 
Therefore, $\xi_{a,1}$ is subcritically Stein fillable. 
\end{proof}

\begin{remark}\label{rem: notsubcriticalfillable}
In each homotopy class of almost contact structures on $S^2\times S^3$ and $S^{2} \widetilde{\times} S^{3}$, there is a unique contact structure which is subcritically Stein fillable \cite[Theorem 1.2]{DGZ}. In view of 
the result of Boyer and Pati \cite[Proposition 3.11]{BP}
and Theorem \ref{thm: fillability}, we deduce that any contact structure $\xi_{a, b}$ (and also $\widetilde{\xi}_{a,b}$) with $b\neq 1$ is not subcritically Stein fillable.
\end{remark}

\subsection{Stein non-fillable contact structures}\label{section: non-fillable}

We will prove Stein non-fillability by contradiction. 
To do this, let us suppose that a Boothby--Wang contact manifold over a closed integral symplectic manifold 
$\Sigma$ admits a Stein filling. 
By capping off the boundary, 
we can construct a closed symplectic manifold containing $\Sigma$ as a symplectic hyperplane section.  
In what follows we first present 
this construction and 
then give a proof of Theorem \ref{thm: fillability} (\ref{item: non-fillable}). 
A similar construction of closed symplectic manifolds can be found in \cite[Section 6.1]{CDvK}. 

Let $(W, d\lambda)$ be a Stein domain whose contact boundary is contactomorphic to the total space of 
the Boothby--Wang bundle 
$$
	p \colon (P, \ker A) \rightarrow (\Sigma, \omega_{\Sigma})
$$
over an integral symplectic manifold $(\Sigma, \ow_{\Sigma})$ with the Euler class equal to a fixed integral lift $-[\omega_{\Sigma}] \in H^{2}(\Sigma; \Z)$. 

Consider the complex line bundle $p \colon E \rightarrow \Sigma$ with the first Chern class $c_1(E) = [\omega_{\Sigma}]$. 
Take a hermitian metric on $E$, and let $\Theta \in \Omega^1(E \setminus \Sigma)$ be a connection 1-form, called the global angular form, with $d\Theta = -p^*\omega_{\Sigma}$. 
Denote the corresponding unit disk bundle by $N \rightarrow \Sigma$.
We may assume that the circle bundle $\partial N \rightarrow \Sigma$ coincides with the Boothby--Wang bundle $P \rightarrow \Sigma$ and that $\Theta|_{\p N} = - A/2\pi$. 
Following \cite[Section 2.2]{BiKh13}, we equip $E$ with the symplectic form $\omega_{E}$ given by
\begin{equation}\label{eq: omega_E}
\omega_{E} =  -d(f(r^2) \Theta) = -df(r^2) \wedge \Theta + f(r^2) p^*\omega_{\Sigma}
\end{equation}
where $r$ denotes the norm on $E$ with respect to the hermitian metric, and $f$ is a positive smooth function such that $f(0) = 1$ and $f' < 0$. 
Note that $\omega_{E}$ is exact away from the zero section with a primitive $-f(r^2) \Theta$.  
The corresponding Liouville vector field is transverse to the boundary $\p N$ of $N$ pointing inwards. 
Up to Liouville homotopy, we may assume that $f(1)A/2\pi = \lda|_{\partial W}$ by \cite[Lemma 4.25]{Gu19}. 
Hence the concave boundary $(\overline{\del N}, -f(1) \Theta|_{\del N})$ is strictly contactomorphic to $(\del W, \lambda|_{\del W})$, where $\overline{\del N}$ denotes $\del N$ with the reversed orientation; see also Remark \ref{rem: orientation} below.
Using Liouville flows, collar neighborhoods of $\del W$ in $W$ and $\overline{\del N}$ in $N$ are identified with 
$$
	((-\epsilon, 0] \times P, d(e^{t+\log f(1)}A)/2\pi) \quad \textrm{and} \quad  ([0, \epsilon) \times P, d(e^{t+\log f(1)}A)/2\pi) 
$$
for some $\epsilon>0$, respectively. 
This allows us to glue $W$ and $N$ together symplectically along their boundary.
Let $(M,\omega)$ denote the glued symplectic manifold $(W,d\lambda) \cup (N, \omega_{E}|_{N})$. 

\begin{remark}\label{rem: orientation}
The contact structure $\ker(-f(r^2)\Theta|_{\del N})$ is negative on $\del N$ when it is oriented as the boundary of $N$. 
A straightforward computation yields 
$$
	(-f(r^2)\Theta|_{\del N}) \wedge (-d(f(r^2)\Theta|_{\del N}))^{n-1} = \frac{1}{n}i_{V} (-d(f(r^2) \Theta))^n, 
$$
where $\dim N=2n-1$ and $V$ is the Liouville vector field which is $-d(f(r^2) \Theta)$-dual to $-f(r^2) \Theta$. 
Since $V$ points inwards along $\del N$, the above equality implies that the given orientation of $\del N$ does not agree with that induced by the contact form. 
As a byproduct, $\ker(-f(r^2)\Theta|_{\del N})$ is positive on $\overline{\del N}$.
\end{remark}

\begin{proposition}\label{prop: gluingandembedded}
The symplectic manifold $(M,\omega)$ contains $(\Sigma, \omega_{\Sigma})$ as 
a symplectic hyperplane section. 
\end{proposition}

\begin{proof}
Since $\Sigma$ is a symplectic submanifold of codimension $2$ and its complement is a Stein manifold, 
it remains to show that $[\omega]=\mathrm{PD}[\Sigma]$. Recall that the symplectic form $\omega$ is given by the symplectic form $d\lda$ on the Stein domain piece $W$ and by the symplectic form $\ow_{E}|_{N}$ on the disk bundle piece. In particular $d\lda$ is exact and hence is vanishing in the cohomology $H_c^2(\Int W)$. In light of the Mayer--Vietoris sequence for the decomposition $M = W \cup N$ which is compatible with taking Poincar\'e dual, it suffices to show that the Poincar\'e dual of the cohomology class $[\ow_{E}] \in H^2_c(E)$ is given by the homology class of the zero section $[\Sigma] \in H_{2n-2}(E)$. In view of \eqref{eq: omega_E}, we see that $[\ow_{E}]$ gives the Thom class of the line bundle $E \rightarrow \Sigma$ as described e.g. in \cite[Equation (6.40)]{BoTu82}. See also \cite[Section 2.2]{BiKh13}. Therefore its Poincar\'e dual is the homology class of the zero section by \cite[Proposition 6.24(b)]{BoTu82}.
\end{proof}

Using Proposition \ref{prop: gluingandembedded}, we prove Theorem \ref{thm: fillability} (\ref{item: non-fillable}).

\begin{proof}[Proof of Theorem \ref{thm: fillability} (\ref{item: non-fillable})]
Suppose on the contrary that $(P_{a, 2}, \xi_{a, 2})$ admits a Stein filling $W$. By Proposition \ref{prop: gluingandembedded}, we can construct a closed integral symplectic $6$-manifold $M$ such that 
the rational ruled surface $(\mathbb P (E_0), \Omega_{a, 2})$ is embedded in $M$ as a symplectic hyperplane section. This contradicts Theorem \ref{thm: embeddability} (\ref{item: non-embedding}) as we have assumed $a \geq 5$. The same argument shows $\widetilde{\xi}_{a, 2}$ with $a \geq 6$ is Stein non-fillable.
\end{proof}

\subsection{Lerman's question}\label{section: almost contact}

Now we shall address Lerman's question on distinguishing contact structures with the same underlying almost contact structure (see Question \ref{q: Lerman} below).

Let us first recall that an \textit{almost contact structure} on an oriented manifold $P$ is 
a cooriented codimension $1$ subbundle $\xi \subset TP$ together 
with a complex bundle structure on it. We say that two almost contact structures $\xi$ and $\xi'$ on $P$ are \textit{equivalent} 
if they can be connected by a combination of smooth homotopies of almost contact structures on $P$ 
and orientation-preserving diffeomorphisms on $P$. 

We deduce the following corollary from Proposition \ref{prop: BW}. 
 
\begin{corollary}\label{cor: equivalencefoalmostcontactstructures}
Let $\xi_{a,b}$ (resp. $\widetilde{\xi}_{a,b}$) be the contact structure on $P_{a,b}$ (resp. $\widetilde{P}_{a,b}$) given as in Section \ref{section: BW}.
Suppose that each pair of $(a,b)$ and $(a',b')$ is coprime. 

\begin{enumerate}
\item $\xi_{a,b}$ and $\xi_{a',b'}$ are equivalent as almost contact structures if and only if $a-b=a'-b'$; 
\item $\widetilde{\xi}_{a,b}$ and $\widetilde{\xi}_{a',b'}$ are equivalent as almost contact structures if and only if $2a-3b=2a'-3b'$.
\end{enumerate}
\end{corollary}

\begin{proof}
According to \cite[Corollary 10]{Ham}, two almost contact structures $\xi_{a,b}$ and $\xi_{a',b'}$ are equivalent 
if and only if $c_{1}(\xi_{a,b})$ and $c_{1}(\xi_{a', b'})$ 
have the same maximal divisibility in the integral cohomology; 
this also applies to $\widetilde{\xi}_{a,b}$ and $\widetilde{\xi}_{a',b'}$.
Therefore the corollary follows from Proposition \ref{prop: BW} immediately. 
\end{proof}

\begin{question}[{Lerman \cite[Question 1]{Ler}}]\label{q: Lerman}
Suppose that Boothby--Wang contact manifolds $(P_{a,b}. \xi_{a,b})$ and $(P_{a',b'}, \xi_{a',b'})$ with $(a,b) \neq (a',b')$ have equivalent underlying almost contact structures. 
Then, are they contactomorphic?
\end{question}

We give a partial negative answer to this question. 

\begin{proposition}\label{prop: Lerman}
Let $(P_{a,b}, \xi_{a,b})$ be the Boothby--Wang contact manifolds over $(S^{2} \times S^{2}, \omega_{a,b})$ as above. 
Then, $\xi_{a,1}$ and $\xi_{a+1,2}$ with $a \geq 4$ are equivalent as almost contact structures but not contactomorphic.
\end{proposition}

\begin{proof}
As $a-1=(a+1)-2$, Corollary \ref{cor: equivalencefoalmostcontactstructures} shows that $\xi_{a,1}$ and $\xi_{a+1,2}$ are equivalent as almost contact structures. 
By Theorem \ref{thm: fillability}, $\xi_{a,1}$ is subcritically Stein fillable whereas $\xi_{a+1,2}$ is not Stein fillable. 
This implies that these contact structures are not contactomorphic. 
\end{proof}

\begin{remark} \label{rem: almostWeinconstrs}
Combining the results in Theorems \ref{thm: embeddability} and \ref{thm: fillability},
we obtain interesting contact structures on $S^2 \times S^3$ (and also on $S^2 \tilde{\times } S^3$) in terms of various fillability: \emph{There are infinitely many  pairwise non-contactomorphic contact structures on $S^2 \times S^3$ which are strongly fillable and almost Weinstein fillable, but not Stein fillable.} Those contact structures are given by the contact structures of the form $\xi_{a, 2}$ on $P_{a, 2}$ with $a \geq 5$. 
Note that the underlying almost contact structure of $\xi_{a, 2}$ is equivalent to the one of $\xi_{a-1, 1}$ by Proposition \ref{prop: Lerman}. 
Furthermore, since the latter is Stein fillable, it follows that $\xi_{a, 2}$ is almost Weinstein fillable. 
Since $\xi_{a,2}$ comes from a Boothby--Wang bundle, the associated disk bundle provides a strong symplectic filling.  
Corollary \ref{cor: equivalencefoalmostcontactstructures} proves that the contact structures $\xi_{a,2}$ have pairwise inequivalent underlying almost contact structures; in particular they are pairwise non-contactomorphic.
For more information on almost  Weinstein fillability, we refer the reader to \cite{BCS}, \cite[Section 2.1.2]{La20} and \cite[Section 1.3]{Zh19}.
\end{remark}

\addcontentsline{toc}{chapter}{Bibliography}
\bibliographystyle{alpha}
\bibliography{S2*S2}

\end{document}